\documentclass[12pt,twoside,english]{article}

\usepackage{amsmath,amsfonts,amssymb,amsthm,color}
\usepackage[a4paper,left=30mm,right=30mm,top=32mm,bottom=30mm]{geometry}
\usepackage[utf8]{inputenc}
\usepackage{babel}
\usepackage[pdfusetitle]{hyperref}
\usepackage{enumitem} 

\newcommand{\R}{\mathbb{R}}
\newcommand{\N}{\mathbb{N}}

\newcommand{\eps}{\varepsilon}
\newcommand{\fhi}{\varphi}

\newcommand{\weakto}{\rightharpoonup}
\newcommand{\weakstar}{\stackrel{\ast}{\rightharpoonup}}
\newcommand{\ra}{\rangle}
\newcommand{\la}{\langle}

\newcommand{\del}{\partial}

\newcommand{\sym}{\mathrm{sym}}
\DeclareMathOperator{\curl}{curl}
\DeclareMathOperator{\dom}{dom}
\DeclareMathOperator{\graph}{graph}

\newcommand{\id}{\mathrm{id}}

\def\calE{\mathcal{E}}
\def\calF{\mathcal{F}}
\def\calG{\mathcal{G}}

\def\calH{\mathcal{H}}

\def\calR{\mathcal{R}}
\def\calQ{\mathcal{Q}}

\def\calW{\mathcal{W}}

\def\frakE{\mathfrak{E}}
\def\Xp{{X}_\delta}

\newcommand\Rnn{\R^{3\times 3}}
\newcommand\Rnns{\Rnn_s}

\newtheorem{theorem}{Theorem}[section]
\newtheorem{lemma}[theorem]{Lemma}
\newtheorem{proposition}[theorem]{Proposition}

\newtheorem{assumption}[theorem]{Assumption}
\newtheorem{definition}[theorem]{Definition}
\theoremstyle{definition}

\numberwithin{equation}{section}

\makeatletter
\def\namedlabel#1#2{\begingroup
   \def\@currentlabel{#2}%
   \phantomsection\label{#1}\endgroup
}
\makeatother

\begin{document}
\pagestyle{myheadings} \markboth{Strain gradient
  visco-plasticity}{Matthias R\"oger and Ben Schweizer}

\thispagestyle{empty}
\begin{center}
  ~\vskip7mm {\Large\bfseries Strain gradient visco-plasticity with
    dislocation\\[4mm] densities contributing to the energy
  }\\[8mm]
  {\large Matthias R\"oger and Ben Schweizer\footnote{Technische
      Universit\"at Dortmund, Fakult\"at f\"ur Mathematik,
      Vogelpothsweg 87, D-44227 Dortmund,
      Germany.  
    }}\\[3mm]
  \medskip

  April 18, 2017
\end{center}

\begin{center}
\vskip5mm
\begin{minipage}[c]{0.86\textwidth}
  {\small {\bfseries Abstract:} We consider the energetic description
    of a visco-plastic evolution and derive an existence result. The
    energies are convex, but not necessarily quadratic. Our model is a
    strain gradient model in which the curl of the plastic strain
    contributes to the energy.  Our existence results are based on a
    time-discretization, the limit procedure relies on Helmholtz
    decompositions and compensated compactness.
    \\[0mm]

    {\bfseries Key-words:} visco-plasticity, strain gradient
    plasticity, energetic solution, div-curl lemma}\\[0mm]

    {\bfseries MSC:} 74C10, 49J45\\[0mm]

\end{minipage}\\[1mm]
\end{center}

\section{Introduction}

The quasi-stationary evolution of a visco-plastic body is analyzed in
an energetic approach.  We use the framework of infinitesimal
plasticity with an additive decomposition of the strain. The equations
are described with three functionals, the elastic energy, the plastic
energy, and the dissipation. The three functionals are assumed to be
convex, but not necessarily quadratic. In this sense, we study a
three-fold non-linear system.

Our interest is to include derivatives of the plastic strain $p$ in
the free energy.  We are therefore dealing with a problem in the
context of strain gradient plasticity.  Of particular importance are
contributions of $\curl(p)$\, to the plastic energy, since this term
measures the density of dislocations. Attributing an energy to plastic
deformations means that hardening of the material is modelled. Since
derivatives of $p$ contribute to the energy, the model introduces a
length scale in the plasticity problem; this is desirable for the
explanation of some experimental results.

We treat a model that was introduced in \cite {MR1761129}, with
analysis available in \cite {GiacominiLussardi}, \cite
{NeffChelAlber}, and \cite {Nesenenko-Neff-2012}. We discuss the
literature below in Section \ref {ssec.literature}. The model is
entirely based on energies and is thermodynamically consistent.  Our
main result regards well-posedness of the system. We use the framework
of energetic solutions to derive an existence result.

\subsection{Model and main results}

We denote the deformation by $u$ and decompose the gradient into an
elastic and a plastic part, $\nabla u = e + p$. We do not use the
symmetrization in the decomposition, but note that only the symmetric
part $\sym(e)$ contributes to the elastic energy.  The elastic energy
$\calW_e$ describes the elastic response of the material and the
plastic energy $\calW_p$ describes hardening effects and gradient
plasticity. The precise form that we study in this contribution is
given in \eqref {eq:We-from-QH}--\eqref {eq:total-en}; in a special
case (setting $H_e = 0$, $H_p = 0$, $r=2$), the two energies read
\begin{align}
  \label{eq:W-sub-e-intro}
  \calW_e(\nabla u, p) &= \int_\Omega Q(\sym(\nabla u - p))  \,,\\
  \label{eq:W-sub-p-intro}
  \calW_p(p) &= \int_\Omega |\curl(p)|^2 + \delta |\nabla p|^2\,,
\end{align}
where $Q$ associates to a symmetric matrix an elastic energy, and
$\delta\ge 0$ is a real parameter. Below, we write the elastic energy
density in the form $W_e(F,p) = Q(\sym(F - p))$. The two energies are
accompanied by a dissipation rate functional $\calR$ with convex dual
$\calR^*$. The latter is used to express the flow rule of the plastic
strain. We do not consider positively $1$-homogeneous functionals
$\calR$ in this work; we hence treat here a visco-plastic model and
not a rate-independent model.

We use the following variables: The deformation $u$ with the two parts
$e$ and $p$ of the gradient.  The stress $\sigma$ depends on elastic
deformations and is, as usual, given by the functional derivative of
the elastic energy. The gradient plasticity model uses one additional
variable, the back-stress variable $\Sigma$. The back-stress $\Sigma$
controls the evolution of the plastic strain and is given by a
functional derivative of the total free energy $\calW = \calW_e +
\calW_p$ with respect to $p$. In its strong form, the plastic
evolution problem reads
\begin{align}
  -\nabla\cdot \sigma &= f\,,\label{eq:u}\\
  \sigma &= \sym\ \nabla_{F} W_e(\nabla u, p)\,,\label{eq:sig}\\
  -\Sigma &\in \del_{p} \calW(\nabla u, p)\,,\label{eq:Sigma}\\
  \del_t p &\in \partial \calR^*(\Sigma)\,.\label{eq:p}
\end{align}
The variational structure of the system can be made even more apparent
by writing the two equations \eqref {eq:u}--\eqref {eq:sig}
equivalently as $f \in \del_u \calW(\nabla u, p)$.

In our results, we treat more general energies than those of \eqref
{eq:W-sub-e-intro}--\eqref {eq:W-sub-p-intro}. The additional term
$H_p(p)$ of \eqref {eq:W-sub-p} allows to introduce more general
hardening laws. The term $H_e(\nabla^s u)$ of \eqref {eq:We-from-QH}
allows to associate an infinite energy to deformations with
self-penetration (we emphasize that we do not introduce growth
assumptions for $H_p$ and $H_e$).

\paragraph{Main results.}

Our main results are existence theorems for the above system of
equations. A major difficulty in the analysis is the non-linear
character of the two equations \eqref {eq:sig} and \eqref {eq:Sigma}
(in combination with the flow rule \eqref {eq:p}, which is always
non-linear). Furthermore, the plastic energy contains the quantity
$\curl(p)$; this means that the back-stress $\Sigma$ contains the
contribution $\curl(\curl(p))$.

The proofs rely on an energetic formulation, avoiding the deformation
variable $u$. This is possible through the use of a marginal
energy, compare for example \cite{RoSa06} and \cite{MielkeRossiSavare}. This concept makes the gradient flow structure of the problem
even more apparent. We construct approximate solutions through a time
discretization of the problem, solving a stationary variational
problem in each time step. At this point, the explicit time dependence
of the energies must be treated with care, since the time dependence
involves the deformation $u$.  The limit procedure relies on
(compensated) compactness properties of the sequence of approximate
solutions. In the case $\delta=0$ we need Helmholtz decompositions and
the div-curl lemma in order to perform the limit procedure. The limit
functions are shown to be energetic solutions to the system.

In our first theorem, we treat the case $\delta>0$ and quite general
energies. In this case, the plastic energy provides estimates for all
derivatives of $p$ and hence compactness of approximating sequences.
Theorem \ref {thm:main-gr} makes the following statement precise:
Given a load $f$, an initial datum $p_0$, a time horizon $T>0$, and
$\Omega\subset \R^3$, there exists an energetic solution to system
\eqref {eq:u}--\eqref {eq:p} on $\Omega\times (0,T)$.

Our second theorem treats the case that the plastic energy contains
$\curl(p)$, but not the full gradient of $p$ (i.e.: $\delta = 0$). In
this case, only estimates for certain derivatives of $p$ are at our
disposal and space-time $L^2$-compactness of approximate solutions
cannot be expected.  Compensated compactness nevertheless allows to
derive our second existence result; loosely speaking, the control of
$\curl(p)$ is dual to the control of $\nabla\cdot \sigma$.  While
quite general energies are treated in Theorem \ref {thm:main-gr}, we
can deal with the case $\delta = 0$ only for certain energies,
essentially those of \eqref {eq:W-sub-e-intro}--\eqref
{eq:W-sub-p-intro}. Theorem \ref {thm:main-eq} states the existence of
solutions for $\delta = 0$.

\subsection{Discussion and comparison with the literature}
\label{ssec.literature}

\paragraph{On the plasticity model.}

The importance of strain gradient models to describe the plastic
deformation of metal is well-known, we mention \cite {Fleck1994475}
for an early model and the discussion of experiments with thin copper
wires. For a comparison of different approaches, see \cite
{Bayley20067268}. The physical basis of a strain gradient model is a
higher order contribution to the energy: Kr\"oner's formula uses the
curl of the plastic strain to measure the dislocation density
(cp.\,\cite {VanGoethem2011} for a recent discussion). Hence $\curl(
p)$ contributes via the self-energy of dislocations to the total energy
\cite {Ortiz-etal-2000}.  For an analytical investigation of the
energy based on single crystal plasticity we refer to \cite
{ContiOrtiz2005}. In comparison to our model we note that the energy
contribution of formula (2.4) in \cite {ContiOrtiz2005} is an
$L^1$-norm of the curl, and not a squared contribution as in \eqref
{eq:W-sub-p-intro}. We regard the energy of \eqref {eq:W-sub-p-intro}
as an approximation that regularizes the mathematically derived
(single crystal) $L^1$-type energy.

The strain gradient model of this work appears e.g.\,in \cite
{MR1761129}. Writing their variables $h^p$, $\alpha^p$,
$\sigma^{\mathrm{dis}}$ as $p$, $\curl(p)$, $\Sigma$, our flow rule
\eqref {eq:p} appears as equation (12), the subdifferential
description of $-\Sigma$ in \eqref {eq:Sigma} appears in Remark 2.2 of
\cite {MR1761129}.  A first mathematical discussion of the resulting
system was performed in \cite {NeffChelAlber}. Existence results were
shown in \cite {GiacominiLussardi} and \cite {Nesenenko-Neff-2012} in
the case of quadratic energies as in \eqref {eq:W-sub-e-intro}--\eqref
{eq:W-sub-p-intro}.  Of a more general nature is the approach of \cite
{MielkeRossiSavare}, which allows to treat also non-convex energies;
we give a more detailed comparison below.

In the context of finite strain elastoplasticity (i.e.\,with a
multiplicative decomposition of the gradient) we are only aware of
models that take the full gradient of $p$ into account in the energy.
An existence result for the single time-step in this situation is
derived in \cite {MielkeMueller2006}.  The time continuous problem was
solved in \cite {MRS16} for multiplicative visco-plasticity, and in
\cite {MR2511255} for multiplicative gradient plasticity. We recall
that the full gradient of $p$ was used in these contributions and note
that quite restrictive growth conditions must be imposed on the energy
contributions.

Let us conclude this discussion of the plasticity model with a remark
concerning the non-quadratic character of the energies (and plasticity
models in general). Plasticity models with an additive decomposition
of the strain are sometimes criticized for the following reason: The
assumption of an infinitesimal deformation does not fit well with the
non-linear character of the plasticity system (in particular, of the
flow rule), since the non-linear character of a system becomes
relevant only at finite deformations. But the difficulty can be
reconciled with a proper rescaling of the system: In the case of small
deformations we consider configurations of the form $\Phi = \id +
\eps\, u$, where $\eps>0$ is a small parameter and $u$ (a quantity of
order $1$) is the rescaled deformation. If, in this scaling, a
non-linear function $\calF$ of the physical equations is of the form
$\calF(\Phi) = \calG((\nabla \Phi-\id)/\eps)$, then the quantity $u$
is described by a non-linear system that involves $\calG$---even in
the limit $\eps\to 0$.

For the above reasons, we are convinced that non-quadratic energies
should be considered in a plasticity problem, even if the framework of
infinitesimal deformations is used.  For the general setting of
plasticity models see \cite {Alber_book98} and \cite {HanReddy99}, we
also mention \cite{Davoli-Francfort-2015} for critical comments
concerning current finite strain plasticity models.

\paragraph{Methods of proof and relations to other mathematical
  results.}

We derive existence results with the help of a time
discretization. The time discrete solutions are found with variational
arguments. The energies provide a priori estimates for the sequence of
approximate solutions and we obtain easily the existence of (weak)
limits. These are our candidates for a solution. The main task is to
perform the limit procedure, i.e.\,to show that the weak limits
provide a solution.

For the limit procedure, we work in the setting of energetic solutions
\cite {MR2182832}. In this approach, a weak solution is defined as a
tuple of functions that satisfies two relations: a stability property
in every time instance $t$ and a (time integrated) energy inequality
(compare our conditions (S$_1$) and (E$_1$) in Definition \ref
{def:solution1}). Since our model uses additionally the back-stress
variable, we have to accompany the solution concept with condition
(F$_1$) to relate $\Sigma$ to the other quantities. We find it
actually helpful to work with an even more condensed system of
equations in which the deformation $u$ is not used explicitely, see
conditions (F$_2$) and (E$_2$) in Definition \ref {def:solution2}. The
limit procedure makes use of lower semicontinuity of functionals and
of compactness properties.

We already mentioned the existence results of \cite
{GiacominiLussardi} and \cite {Nesenenko-Neff-2012}. These results are
also based on time discretizations of the problem.  Since only
quadratic energies are studied in both \cite {GiacominiLussardi} and
\cite {Nesenenko-Neff-2012}, the two relations \eqref {eq:sig} and
\eqref {eq:Sigma} are linear in their case. This allows for a much
more direct derivation of the limit equations.

Let us compare our results once more with those of \cite
{MielkeRossiSavare}. Their system can be related to our model by
replacing their variables $\Phi$, $z$, $I$, $\calE^1$, $I^2$, $\psi$
by our variables $u$, $p$, $W$, $W_p$, $W_e$, $R$.  We have less
restrictive assumptions on $Q$ (compare the square-root growth
assumption on $D_z W$ in (W$_3$) of \cite
{MielkeRossiSavare}). Another difference is that in \cite
{MielkeRossiSavare} the values of $z$ (our $p$) are confined to a
compact set by using an indicator function in the energy. Furthermore,
their exponent $r$ in $|\nabla p|^r$ must be larger than the space
dimension.  We finally recall that \cite {MielkeRossiSavare} always
uses the full gradient of $z \sim p$, while we treat also the case
$\delta=0$.  On the other hand: We demand the convexity of all
energies. In this point the contribution \cite {MielkeRossiSavare}
treats a much more general setting.

For other results containing the full gradient (the case $\delta >0$),
we mention the book by Roubicek \cite{Roubicek-book} and note that the
explicit time dependence of the energy (that is present in our model
due to the dependence of $\calE_1$ on $f$) is not covered in his
result ($\del_t f \equiv 0$ in our setting). The same remark is valid
concerning the abstract result of Colli and Visintin \cite{MR1070845}.

\medskip The remainder of this article is organized as follows. In
Section \ref {sec.model-ass} we state the model in a precise way and
formulate the assumptions on the data.  In Section \ref
{sec.stability} we show a stability property of our solution concept:
Under certain assumptions, sequences of approximate solutions converge
to solutions.  In Section \ref {sec:tds} we perform the time
discretization and construct approximate solutions. The stability
results of Section \ref {sec.stability} can be applied and yield that
the limit functions provide a solution to the time-continuous model.

\section{Model equations and assumptions}
\label{sec.model-ass}

For a domain $\Omega\subset \R^3$ and a space-time cylinder $\Omega_T
:= \Omega\times (0,T)$ we denote the deformation by $u :\Omega_T\to
\R^3$ and the plastic strain by $p:\Omega_T\to \Rnn$. Following \cite
{MR1761129} and \cite {NeffChelAlber}, we do not impose that the
plastic strain is symmetric; to incorporate rotational invariance of
the energy, the elastic energy depends only on $e = \sym (\nabla u -
p)$. We write here $\Rnns$ for the space of symmetric $n\times n$
matrices and denote the projection onto symmetric matrices by $\sym :
\Rnn \to \Rnns$, $F\mapsto \frac12 (F + F^T)$. We use also the
notation $\nabla^s u := \sym\, \nabla u$ such that $\nabla^s u = e +
\sym(p)$.

Despite its importance in mechanics, for simplicity of the
presentation, we do not incorporate the decomposition into spherical
and deviatoric parts in our model.

We use an energetic approach and formulate the plasticity equations
with energies and dissipation rate functionals. We use an elastic
energy $\calW_e$, a plastic energy $\calW_p$, and the total energy
$\calW = \calW_e + \calW_p$.  The energies are based on energy density
functions. We use the following four functions:
\begin{equation}
  \label{eq:QHKR}
  \begin{split}
    &Q : \Rnns \to \R,\ e\mapsto Q(e)\,,\qquad
    R : \Rnn \to \R,\ q\mapsto R(q)\,,\\
    &H_e : \Rnns \to \R,\ \nabla^s u\mapsto H_e(\nabla^s u) ,\qquad 
    H_p : \Rnn \to \R,\ p\mapsto H_p(p) \,.
  \end{split}  
\end{equation}
The elastic energy density $W_e: \Rnn \times \Rnn \to \R$ is
\begin{equation}
  W_e(\nabla u, p) :=  Q(\sym(\nabla u - p)) + H_e(\nabla^s u) \,.
  \label{eq:We-from-QH}
\end{equation}

We are interested in a strain gradient plasticity model. To introduce
derivatives of $p$ in the energies, we use a factor $\delta\ge 0$ and
an exponent $r\in \R$, $r> \frac{6}{5}$. We furthermore consider the
curl of $p = (p_{i,j})_{i,j}$, which we define row-wise: $(\curl
p)_{k,.} := \curl (p_{k,.})$.  With this definition, every smooth
function $\fhi:\Omega\to \R^3$ satisfies $(\curl \nabla\fhi)_{k,.}  =
\curl ((\nabla\fhi)_{k,.}) = \curl \nabla\fhi_k = 0$.  For functions $u:
\Omega\to\R^{3}$, $p: \Omega\to\Rnn$ we define the energies
\begin{align}
  \label{eq:W-sub-e}
  \calW_e(\nabla u, p) &:= \int_\Omega W_e(\nabla u, p) \,,\displaybreak[3]\\
  \label{eq:W-sub-p}
  \calW_p(p) &:=  \int_\Omega \left\{H_p(p) 
    + |\curl\, p|^2 + \delta |\nabla p|^r \right\}\,,\displaybreak[3]\\
  \calR(q) &:= \int_\Omega R(q)\,,
\end{align}
and use the two convex duals $R^*$ and $\calR^*$. The total energy is
\begin{equation}
  \label{eq:total-en}
  \calW(\nabla u, p) := \calW_e(\nabla u, p) +
  \calW_p(p)\,.  
\end{equation}

\subsubsection* {General assumptions} We next collect our assumptions
on the energy functionals, on initial and boundary conditions, and on
the applied loads.

\begin{assumption}[Energy and dissipation
  functional]\label{ass:general}
  Let $Q: \Rnns \to [0,\infty)$ and $R, R^*:\Rnn \to [0,\infty)$ be
  convex and continuous, and let $H_e: \Rnns \to [0,\infty]$ and $H_p :
  \Rnn \to [0,\infty]$ be convex, proper and lower-semicontinuous,
  with $0\in \dom(H_e)$.

  We assume that $Q,R,R^*$ have quadratic growth: For $0<c<C$ holds
  \begin{equation}
    c |\zeta|^2 \,\leq\, Q(\zeta)\,\leq\, C
    |\zeta|^2   \,,\label{eq:cond-QR-1}\\
  \end{equation}
  for all matrices $\zeta\in \Rnn_s$, and
  \begin{align}
    &c |\xi|^2 \,\leq\, R(\xi)\quad \text{ and }\quad 
    R^*(\xi)\,\leq\, C |\xi|^2\,,\label{eq:cond-QR-2}\\ 
    &c |\xi|^2 \,\leq\, R^*(\xi)\quad \text{ and }\quad  
    R(\xi)\,\leq\, C |\xi|^2\,,\label{eq:cond-QR-3}
  \end{align}
  for all matrices $\xi\in \Rnn$. 
\end{assumption}

\paragraph {Function spaces and boundary conditions.}
We always assume that $\Omega\subset \R^3$ is a bounded
Lipschitz-domain. In the case $\delta=0$, we assume additionally that
$\Omega$ is simply connected with connected boundary, and that
$\del\Omega$ is of class $C^{1,1}$ (this assumption could be replaced
by a convexity requirement). We always assume that we are given a
relatively open non-empty subset $\Gamma_D\subset \del\Omega$.

To impose a Dirichlet boundary condition on $\Gamma_D$ we introduce
the function space $H^{1}_D(\Omega; \R^3)$ of functions in
$H^1(\Omega)$ with vanishing trace on $\Gamma_D$. Its dual space is
denoted as $H^{-1}_D(\Omega;\R^3)$.  For the prescribed load we assume
\begin{align}
  f\in L^2(0,T;L^{2}(\Omega;\R^3)) \cap
  H^1(0,T;H^{-1}_D(\Omega;\R^3))\,. \label{ass:f}
\end{align}
For $p$ we consider an initial condition 
\begin{align}
  p_0\in L^2(\Omega;\Rnn)\,. \label{ass:p0}
\end{align}

We denote by $H_0(\Omega,\curl)$ the space of all $p\in L^2(\Omega;
\Rnn)$ that have zero tangential trace $p\times\nu$ (where $\nu$ is
the exterior normal on $\del\Omega$) in the sense that
\begin{gather}
  \int_\Omega \big(\varphi : \curl p - \curl\varphi : p \big) \,=\,
  0\quad\text{ for all }\varphi\in
  C^1(\bar\Omega;\Rnn)\,. \label {eq:bdry-p}
\end{gather}
We equip this space with the inner product
\begin{align*}
  (p,q)_{H_0(\Omega,\curl)} \,=\, (p,q)_{L^2(\Omega;\Rnn)} + (\curl
  p,\curl q)_{L^2(\Omega;\Rnn)}
\end{align*}
and the induced norm $\|\cdot\|_{H_0(\Omega,\curl)}$. Then
$H_0(\Omega,\curl)$ is a Hilbert space and coincides with the
completion of $C^\infty_c(\Omega;\Rnn)$ with respect to
$\|\cdot\|_{H_0(\Omega,\curl)}$, see \cite[Theorem 2.6]{GiRa79}.

We note that Assumption \ref {ass:general} implies that $\calR,
\calR^*:L^2(\Omega; \Rnn)\to [0,\infty)$ and the function
$\calQ:L^2(\Omega; \Rnns)\to [0,\infty)$ of \eqref {eq:calQ} are
continuous.

\paragraph{The strong formulation.}

In the strong formulation, we seek functions $u,\sigma,p,\Sigma$ that
satisfy the non-linear system \eqref {eq:u}--\eqref {eq:p} in a
classical sense.  Let us collect the equations in the special case
$H_e \equiv 0$, $H_p \equiv 0$, using $e = \sym(\nabla u - p)$.
Equation \eqref{eq:sig} becomes $\sigma = \sym\ \nabla_{F} Q(e)$, the
standard equation for the stress.  We note that, strictly speaking,
this formula does not need the symbol ``$\sym$'': The gradient of $Q$
on the space of symmetric matrices is automatically a symmetric
matrix.  In the same setting, \eqref {eq:Sigma} becomes $\Sigma =
\nabla_{F} Q(\nabla^s u - p) - \nabla_{p} \calW_p(p) = \sigma - Lp$
with the positive differential operator $L$ which reads $L p = 2
\curl\curl(p) - \delta \Delta_r p$.

\subsubsection*{Weak formulation in the variables $(u,p,\Sigma)$}

\begin{definition}[Weak solution in primary
  variables]\label{def:solution1}
  We call $(u,p,\Sigma)$ a weak solution of \eqref{eq:u}--\eqref
  {eq:p} (with boundary conditions) iff the following holds:
  \begin{enumerate}
  \item[(R$_1$) \namedlabel{it:R1}{(R$_1$)}] 
    Regularity and boundary conditions:
    \begin{align}
      \label{eq:spaces-soln}
      &u\in L^2(0,T;H^1_D(\Omega;\R^3))\,,\quad 
      \Sigma\in L^2(0,T;L^2(\Omega; \Rnn))\,,\\
      &p\in H^{1}(0,T;L^2(\Omega; \Rnn)) \cap L^2(0,T;H_0(\Omega,\curl))\,,\\
      &p\in L^\infty(0,T;W^{1,r}(\Omega))\quad\text{ if } \delta>0\,,
    \end{align}
    and $p|_{t=0} = p_0$.
  \item[(S$_1$) \namedlabel{it:S1}{(S$_1$)}] Pointwise energy
    minimization: Equations \eqref{eq:u}--\eqref {eq:sig} are
    satisfied in the following sense: for a.e. $t\in (0,T)$ there
    holds
    \begin{align}
      \label{eq:energy-min}
      \int_{\Omega} W_e(\nabla u(t), p(t)) -
      \int_{\Omega} f(t)\cdot u(t)
      \le  \int_{\Omega} W_e(\nabla \fhi, p(t)) 
      -  \int_{\Omega} f(t)\cdot \fhi
    \end{align}
    for every $\fhi\in H^1_D(\Omega;\R^3)$.
  \item[(F$_1$) \namedlabel{it:F1}{(F$_1$)}] Back-stress variable:
    Instead of \eqref {eq:Sigma} we demand (here, the dual is formed
    with respect to the variable $p$, the argument $\nabla u$ is
    fixed): Almost everywhere in $(0,T)$ holds
    \begin{equation}\label{eq:F1}
      \calW(\nabla u, p) + \calW^*(\nabla u, -\Sigma) =  
      \la -\Sigma, p \ra\,.
    \end{equation}
  \item[(E$_1$) \namedlabel{it:E1}{(E$_1$)}] Energy inequality:
    Instead of the flow rule \eqref {eq:p} we demand that for
    a.e. $t\in (0,T)$ holds
    \begin{align}
      \notag
      & \left[\calW(\nabla u(s), p(s)) -\int_\Omega f(s)\cdot
        u(s)\right]_{s=0}^t + \int_0^t \left\{
        \calR(\del_t p(s)) +
        \calR^*(\Sigma(s)) \right\}\, ds\\
      &\qquad \le -\int_0^t \langle \del_t f(s),u(s)\rangle\,
      ds\,.
      \label{eq:energy-ineq-first}
    \end{align}
  \end{enumerate}
\end{definition}

\subsubsection*{Weak formulation in the variables $(p,\Sigma)$}

Given an external load $f\in H^{-1}_D(\Omega;\R^3)$ we define two marginal
functionals that assign to $p\in L^2(\Omega;\Rnn)$ an energy:
\begin{align}
  \label{eq:reduced-energy-1}
  \calE_1(p;f) &:= \inf \left\{ \calW_e(\nabla \fhi, p) -
    \langle f, \fhi \rangle\Big|\, \fhi\in
    H^1_D(\Omega;\R^3) \right\}\,,\\
  \calE(p;f) &:= \calE_1(p;f) + \calW_p(p)\,.
  \label{eq:reduced-energy-total}
\end{align}
We denote by $\frakE(p,f)$ the set of all $u\in H^1_D(\Omega;\R^3)$
that attain the infimum in \eqref{eq:reduced-energy-1}.  This set is
non-empty, see Lemma \ref{lem:marginal}.

With the marginal functionals $\calE_1$ and $\calE$ we can give an
alternative formulation of the equations in the variables
$(p,\Sigma)$, avoiding $u$.  This reduction is based on the fact that
requirement \ref{it:S1} is equivalent to: For almost every $t\in
(0,T)$ holds the energy minimization property
\begin{align}
  \label{eq:weak-energy-min-note}
  &\int_{\Omega} W_e(\nabla u(t), p(t)) - \int_{\Omega} f(t)\cdot u(t)
  = \calE_1(p(t); f(t))\,.
\end{align}

\begin{definition}[Weak solution of the condensed
  system]\label{def:solution2}
  We call $(p,\Sigma)$ a weak solution of \eqref{eq:u}--\eqref {eq:p}
  (with boundary conditions) iff the following holds:
  \begin{enumerate}
  \item[(R$_2$) \namedlabel{it:R2}{(R$_2$)}] On $p$ and $\Sigma$ we
    demand the properties of Item \ref{it:R1}.
  \item[(F$_2$) \namedlabel{it:F2}{(F$_2$)}] Back-stress variable:
    Almost everywhere in $(0,T)$ holds
    \begin{equation}\label{eq:F2}
       \calE\big(p;f\big) + \calE^*\big(-\Sigma; f\big)  =  
       \la -\Sigma, p \ra\,.
    \end{equation}
  \item[(E$_2$) \namedlabel{it:E2}{(E$_2$)}] Energy inequality: For
    almost all $t\in (0,T)$ holds
    \begin{equation}
      \label{eq:E2}
      \begin{split}
        &\left[ \calE\big(p(s);f(s)\big)\right]_{s=0}^t +
        \int_0^t \left\{ \calR(\del_t p(s)) + 
          \calR^*(\Sigma(s)) \right\}\, ds\\
        &\qquad \le -\int_0^t \inf_{\tilde u\in \frakE(p(s),f(s))}
        \big\la \partial_t f(s), \tilde u\big\ra\, ds\, .        
      \end{split}      
    \end{equation}
  \end{enumerate}
\end{definition}
The (negative of the) integrand on the right-hand side of
\eqref{eq:E2} corresponds to a generalized time-derivative of the
functional $\varphi\mapsto \calW(\nabla\varphi,p(s)) - \la
f(s),\varphi\ra$, compare Example 3 in the introduction in
\cite{MielkeRossiSavare}.

Formally, the three solution concepts are equivalent. Rigorous
statements are collected in the following proposition.

\begin{proposition}[Equivalence of solution
  concepts]\label{prop:solution_concepts}
  The following holds:
  \begin{enumerate}
  \item Let $(p, \Sigma)$ be a solution according to Definition \ref
    {def:solution2}. Furthermore, let us assume that for almost all
    $t\in (0,T)$ the set $ \frakE(p(t),\Sigma(t))$ consists of a
    unique element $u(t)$.  Then $u\in L^2(0,T;H^1_D(\Omega;\R^3))$
    holds and the triple $(u,p,\Sigma)$ is a solution according to
    Definition \ref {def:solution1}.
  \item Let $(u,p,\Sigma)$ be a weak solution according to Definition
    \ref{def:solution1}. Let $Q$, $H_e$, and $H_p$ be
    Fr{\'e}chet-differentiable and let the solution have the
    regularity $\partial_t p \in L^2(0,T;L^2(\Omega;\Rnn))$, $p\in
    L^2(0,T;H^2(\Omega;\Rnn))$, $u\in L^2(0,T;H^2(\Omega;\R^3))$, and
    $\Sigma\in L^2(0,T; H^1(\Omega; \Rnn))$. Then, with $\sigma$
    defined by \eqref{eq:sig}, $(u,\sigma,p,\Sigma)$ is a strong
    solution to \eqref{eq:u}--\eqref{eq:p}.
  \item Let $(u,\sigma,p,\Sigma)$ be a strong solution according to
    \eqref{eq:u}--\eqref{eq:p} and let $Q$, $H_e$, $H_p$ be
    Fr{\'e}chet-differentiable. Then $(u,p,\Sigma)$ is a weak solution
    according to Definition \ref{def:solution1}.
  \end{enumerate}
\end{proposition}

The proof of Proposition \ref{prop:solution_concepts} is given in
Appendix \ref{sec.marginal-solutions}. We note that the assumption on
the uniqueness of the minimizer in Item 1 is only used to ensure that
there exists a measurable selection $t\mapsto u(t)\in
\frakE(p(t),\Sigma(t))$.

\paragraph{Main results.} We now formulate our main results concerning
the existence of energetic solutions.

\begin{theorem}[Existence result for $\delta>0$]\label{thm:main-gr}
  Let Assumption \ref{ass:general} on the energies be satisfied and
  let $f$ and $p_0$ be as in \eqref {ass:f}--\eqref {ass:p0}. For
  $\delta>0$, there exists a weak solution $(p,\Sigma)$ to system
  \eqref {eq:u}--\eqref {eq:p} in the sense of Definition \ref
  {def:solution2}.
\end{theorem}

In the case $\delta=0$ we must restrict ourselfes to energies with
$H_e\equiv 0$ and $H_p\equiv 0$.

\begin{theorem}[Existence result for $\delta=0$]\label{thm:main-eq}
  Let Assumption \ref{ass:general} on the energies and let the
  assumptions on $f$, $p_0$, and $\Omega$ be satisfied. We consider
  the case $\delta=0$ and the energies of \eqref {eq:W-sub-e}--\eqref
  {eq:total-en} with $H_e\equiv 0$ and $H_p\equiv 0$. There exists a
  weak solution $(p,\Sigma)$ to system \eqref {eq:u}--\eqref {eq:p} in
  the sense of Definition \ref {def:solution2}.
\end{theorem}

\section{Stability results}
\label{sec.stability}

This section is devoted to stability results for the plasticity system
in the weak form: We show that, for convergent sequences of
approximate solutions, the limit functions are solutions to the
plasticity system \eqref{eq:u}--\eqref{eq:p}. 

We work with two families of functions, one denoted with a hat and the
other denoted with an overbar. Later on, the first will be a sequence
that is obtained by a piecewise affine interpolation of a
time-discrete approximate solution, the latter will be the
corresponding piecewise constant interpolant.

\subsection{Stability for $\delta>0$}

The boundedness and the convergence properties of the sequences are
summarized in the following assumption.

\begin{assumption}[Convergence properties,
  $\delta>0$]\label{ass:convergence}
  We consider a sequence of approximate solutions $(\bar p^N,
  \bar\Sigma^N)_N$ together with a sequence of loads $(\bar f^N)_N$
  that satisfy
  \begin{align}
    \bar p^N \,&\weakto\, p \quad\text{ in }
    L^r(0,T;W^{1,r}(\Omega,\Rnn)\cap H_0(\Omega,\curl))\,,
    \label{eq:pN}\\
    \bar p^N \,&\to\, p \quad\text{ in }L^2(0,T;L^2(\Omega,\Rnn))\,, 
    \label{eq:strong-pN-bar}\\
    \bar \Sigma^N \,&\weakto\, \Sigma \quad\text{ in }L^2(0,T;L^2(\Omega,\Rnn))\,, 
    \label{eq:SigmaN}\\
    \bar f^N\,&\to\, f \quad\text{ in }L^2(0,T;L^2(\Omega;\R^3))\,,
    \label{eq:fN-bar}
  \end{align}
  as $N\to \infty$. We furthermore assume that there are
  approximations $(\hat p^N)_N$ and $(\hat f^N)_N$ that satisfy, as
  $N\to \infty$,
  \begin{align}
    \hat p^N \,&\weakto\, p \quad\text{ in }H^1(0,T;L^2(\Omega,\Rnn))\,, 
    \label{eq:dtpN-hat}\\
    \hat p^N \,&\to\, p \quad\text{ in }L^2(0,T;L^2(\Omega,\Rnn))\,, 
    \label{eq:strong-pN-hat}\\
    \hat f^N\,&\to\, f \quad\text{ in }H^1(0,T;H^{-1}_D(\Omega;\R^3)) \,.
    \label{eq:fN-hat}
  \end{align}
  Moreover, we assume that the two sequences $(\hat p^N)_N$ and $(\bar
  p^N)_N$ are bounded in the space $L^\infty(0,T;
  H_0(\Omega,\curl)\cap W^{1,r}(\Omega;\Rnn))$ and that $(\bar f^N)_N$
  is bounded in $L^\infty(0,T; H_D^{-1}(\Omega))$. Regarding initial
  values we assume that $\hat p^N(0)=\bar p^N(0)=p(0)=p_0$ and $\hat
  f^N(0) = \bar f^N(0) = f(0)$ holds for all $N\in\N$.  Finally, to
  avoid issues regarding measurability, we assume that $\bar p^N$,
  $\bar f^N$, and $\del_t \hat f^N$ are simple functions for all
  $N\in\N$.
\end{assumption}

The next assumption expresses that the functions are approximate
solutions.

\begin{assumption}[Approximate solution
  properties]\label{ass:approx-sol}
  Let $(\bar p^N, \bar\Sigma^N)_N$ together with $(\hat p^N)_N$ and
  $(\bar f^N)_N$, $(\hat f^N)_N$ be sequences as in Assumption \ref
  {ass:convergence}.  We assume that these functions are approximate
  solutions in the following sense:
  \begin{enumerate}
  \item Relation \eqref{eq:F2} of Item \ref{it:F2} is approximately
    satisfied: For almost every $t\in (0,T)$ holds
    \begin{equation}\label{eq:approx-F2}
      \calE\big(\bar p^N(t);\bar f^N(t)\big) 
      + \calE^*\big(-\bar \Sigma^N(t);\bar f^N(t)\big) \le  
      \la -\bar \Sigma^N(t), \bar p^N(t) \ra + g_N(t)\,,
    \end{equation}
    where the error functions $g_N$ satisfy $\int_0^T |g_N(t)|\, dt\to
    0$ as $N\to \infty$.
  \item The energy inequality \eqref{eq:E2} of Item \ref{it:E2} is
    approximately satisfied: For almost every $t\in (0,T)$, as $N\to \infty$,
    there holds
    \begin{align}
      &\calE\big(\bar p^N(s); \bar f^N(s)\big) \Big|_{s=0}^t
      + \int_0^t \left\{ \calR(\del_t \hat p^N) +
        \calR^*(\bar \Sigma^N) \right\}\, ds \notag\\
      &\qquad \le -\int_0^t \inf_{\tilde u^N\in
        \frakE(\bar p^N(s),\bar f^N(s))}\big\la \partial_t \hat f^N(s), \tilde
      u^N\big\ra\, ds + o(1)\,. \label{eq:approx-E2}
    \end{align}
  \end{enumerate}
\end{assumption}

\begin{proposition}\label{prop:delta_>0}
  Let Assumptions \ref{ass:convergence} and \ref{ass:approx-sol} be
  satisfied. Then the pair $(p, \Sigma)$ is a weak solution of the
  original problem in the sense of Definition \ref{def:solution2}.
\end{proposition}

\begin{proof}
  As weak limits, the functions $p$ and $\Sigma$ are in the
  appropriate function spaces of Item \ref{it:R2}: $\Sigma\in
  L^2(0,T;L^2(\Omega; \Rnn))$, $p\in H^{1}(0,T;L^2(\Omega; \Rnn)) \cap
  L^2(0,T;H_0(\Omega,\curl))$, $p\in L^\infty(0,T;W^{1,r}(\Omega))$.
  The properties of Items \ref{it:F2} and \ref{it:E2} are proved in
  the subsequent lemmas. Together, \ref{it:R2}, \ref{it:F2},
  \ref{it:E2} imply that $(p, \Sigma)$ is a weak solution.
\end{proof}

\begin{lemma}[Item \ref{it:F2}]\label{lem:item-F2_>0}
  Let Assumptions \ref{ass:convergence} and \ref{ass:approx-sol} be
  satisfied. Then Item \ref{it:F2} holds for the limit functions
  $(p,\Sigma)$.
\end{lemma}

\begin{proof} We use an arbitrary non-negative function $\theta\in
  C^\infty(0,T; \R)$ and consider inequality \eqref{eq:approx-F2} in
  the integrated form
  \begin{align*}
    0\,&\geq\, \liminf_{N\to\infty} \int_0^T \theta(t)
    \Big[ \calE\big(\bar p^N(t); \bar f^N(t)\big) +
    \calE^*\big(-\bar \Sigma^N(t); \bar f^N(t)\big) 
    + \la \bar \Sigma^N(t), \bar p^N(t)\ra\Big]\,dt\,.
  \end{align*}
  By the definition of the convex dual $\calE^*$, for an arbitrary
  $\eta\in W^{1,r}(\Omega;\Rnn)\cap H_0(\Omega,\curl)$, we have
  \begin{align*}
    0\,&\geq\, \liminf_{N\to\infty} \int_0^T \theta(t)
    \Big[\calE\big(\bar p^N(t);\bar f^N(t)\big) 
    + \la -\bar \Sigma^N(t),\eta-\bar p^N(t)\ra 
    -\calE\big(\eta;\bar f^N(t)\big)\Big]\,dt\,.
  \end{align*}
  We can exploit the strong convergence of $\bar p^N$ and $\bar f^N$
  and the Lipschitz property of Item \ref{it:lem2.5-4} of Lemma
  \ref{lem:marginal} to conclude the convergence of the last two
  terms. In the first term we use once more the Lipschitz property of
  $\calE_1$ (we exploit the bounds for $\bar f^N \in L^\infty(0,T;
  H_D^{-1}(\Omega))$ and $\bar p^N\in L^\infty(0,T;L^2(\Omega))$ to
  control the Lipschitz constant) and the lower semicontinuity of
  $\calW_p$ and obtain
  \begin{align}
    0\, &\geq\, \int_0^T \theta(t)\Big[\calE\big(p(t);f(t)\big) + \la
    -\Sigma(t),\eta-p(t)\ra -\calE\big(\eta;f(t)\big)\Big]\,dt\,. \label{eq:F2-eta-pre}
  \end{align}

  In order to localize in $t$, we proceed as follows. We consider a
  countable dense subset
  $\big\{\big(\eta_i,\calW_p(\eta_i)\big):\eta_i\in \dom(\calW_p),
  i\in\N\big\}$ of $\graph \calW_p|_{\dom(\calW_p)}$, which is
  separable as a subset of the separable space
  $L^2(\Omega;\Rnn)\times\R$.
  
  Since $\theta$ was arbitrary, we deduce from \eqref{eq:F2-eta-pre}
  that there exists an exceptional set of time instances $B\subset
  (0,T)$ of measure zero such that, for all $t\in (0,T)\setminus B$
  and all $i\in\N$,
  \begin{align}
    0 \,\geq &\, \calE\big(p(t);f(t)\big) + \la
    -\Sigma(t),\eta_i-p(t)\ra -\calE\big(\eta_i;f(t)\big)\,. 
    \label{eq:F2-eta-i}
  \end{align}
  
  For an arbitrary $\eta\in \dom(\calW_p)$ there exists a subsequence
  $i\to\infty$ (not relabeled) such that $\eta_i\to \eta$ in
  $L^2(\Omega;\Rnn)$ and $\calW_p(\eta_i)\to\calW_p(\eta)$. Using the
  first property, we can pass to the limit in the second term on the
  right hand side of \eqref{eq:F2-eta-i} and also, by the $2$-growth
  of $Q$, in $\calE_1(\eta_i;f(t))$. Using the second property yields,
  for all $t\in (0,T)\setminus B$,
  \begin{align*}
    0\,\geq &\, \calE\big(p(t);f(t)\big) + \la -\Sigma(t),\eta -p(t)
    \ra -\calE\big(\eta;f(t)\big)\,,
  \end{align*}
  first for all $\eta\in \dom(\calW_p)$ and then, since
  $\calE(\cdot,f(t))$ is infinite outside $\dom(\calW_p)$, for all
  $\eta\in L^2(\Omega;\Rnn)$.  Taking the supremum over all $\eta\in
  L^2(\Omega;\Rnn)$ yields
  \begin{align*}
    \la -\Sigma(t),p(t) \ra\,\geq &\, \calE\big(p(t);f(t)\big) +
    \calE^*\big( -\Sigma(t);f(t)\big)
  \end{align*}
  for all $t\in (0,T)\setminus B$, and hence the claim.
\end{proof}

\begin{lemma}[Reconstruction of displacement fields]\label{lem:3.5}
  Let Assumption \ref{ass:convergence} be satisfied.  There exists a
  sequence $(\bar u^N)_N$ in $L^2(0,T;H^1_D(\Omega;\R^3))$ such that
  \begin{align}
    \bar u^N(t)\,&\in\, \frakE(\bar p^N(t);\bar f^N(t))
    \quad\text{and}\notag\\
    \la \partial_t \hat f^N(t), \bar u^N(t)\ra \,&=\,
    \inf_{\tilde u\in \frakE(\bar p^N(t), \bar f^N(t))}
    \big\la \partial_t \hat f^N(t), \tilde u\big\ra
    \label{eq:lem-min-uN}
  \end{align}
  for every $N\in\N$ and almost every $t\in (0,T)$.  Furthermore,
  there exists a function $u\in L^2(0,T;H^1_D(\Omega;\R^3))$ and a
  subsequence $N\to \infty$ such that
  \begin{align}
    \bar u^N\,&\weakto\, u\quad\text{ in } L^2(0,T;H^1_D(\Omega;\R^3))
    \quad\text{ as }N\to\infty, \label{eq:lem-conv-uN}\\
    u(t)\,&\in\, \frakE(p(t);f(t)) \quad\text{ for almost every }t\in (0,T). \label{eq:lem-min-u}
  \end{align} 
\end{lemma}

\begin{proof}
  By Lemma \ref{lem:marginal}, Item \ref{it:lem2.5-3}, and the
  continuity of $\la \partial_t \hat f^N(t), \cdot\ra$ under weak
  convergence in $H^1_D(\Omega;\R^3)$, for any $N\in\N$ and almost
  every $t\in [0,T]$, there exists a function $\bar u^N(t)\in
  \frakE(\bar p^N(t); \bar f^N(t))$ that satisfies the minimizing
  property \eqref{eq:lem-min-uN}. Note that we can choose simple
  functions $t\mapsto \bar u^N(t)$; in particular, the functions are
  measurable.  Lemma \ref{lem:marginal}, Item \ref{it:lem2.5-3}, and
  Assumption \ref{ass:convergence} imply that $(\bar u^N)_N$ is
  uniformly bounded in $L^2(0,T; H^1_D(\Omega;\R^3))$, hence we deduce
  the existence of $u\in L^2(0,T;H^1_D(\Omega;\R^3))$ such that
  \eqref{eq:lem-conv-uN} holds.

  We next want to verify the minimizing property of the limit function
  $u$. Let $\fhi\in H^1_D(\Omega;\R^3)$ be arbitrary. For any
  non-negative $\theta\in C^\infty_c(0,T; \R)$ there holds
  \begin{align*}
    &\int_0^T \theta(t)\Big(\int_{\Omega} W_e(\nabla u(t), p(t))\,  -
     \int_{\Omega} f(t)\cdot  u(t)\,\Big) dt\\
    &\qquad \le \liminf_N 
    \int_0^T \theta(t)\Big(\int_{\Omega} W_e(\nabla \bar u^N(t), \bar p^N(t))\,  -
    \int_{\Omega} \bar f^N(t)\cdot \bar u^N(t)\,\Big) dt\\
    &\qquad \le
    \liminf_N  \int_0^T \theta(t)
    \Big(\int_{\Omega} W_e(\nabla \fhi, \bar p^N(t))\, 
    -   \int_{\Omega} \bar f^N(t)\cdot \fhi\,\Big) dt \\
    &\qquad =  \int_0^T \theta(t)\Big(\int_{\Omega} W_e(\nabla \fhi, p(t))\,  
    -   \int_{\Omega} f(t)\cdot \fhi\,\Big) dt\,,
  \end{align*}
  where we used convexity of $W_e$ and weak convergence of $\bar u^N$
  in the first inquality, the minimization property of $\bar u^N$ in
  the second inequality, and the strong convergence of $\bar p^N$
  together with the 2-growth of $Q$ in the last equality.  Since
  $\theta$ was arbitrary, this shows, for almost every $t\in (0,T)$,
  \begin{align}
    & \int_{\Omega} W_e(\nabla u(t), p(t))\,  -
     \int_{\Omega} f(t)\cdot  u(t)\,
     \leq\,\int_{\Omega} W_e(\nabla \fhi - p(t))\,  
    -   \int_{\Omega} f(t)\cdot \fhi\,. \label{eq:dpp-S1-pre}
  \end{align}
  We next argue as in the proof of Lemma \ref{lem:item-F2_>0} and show
  that the last inequality holds for almost all $t\in (0,T)$ and all
  $\fhi\in H^1_D(\Omega;\R^3)$ (i.e.: the set of admissible $t$'s can
  be chosen independent of $\fhi$). We define the functional $\calH_e:
  H^1_D(\Omega;\R^3)\to [0,\infty]$,
  \begin{align*}
    \calH_e(\fhi)\,=\, \int_\Omega H_e(\nabla^s\fhi(x))\,dx\,,
  \end{align*}
  and choose a dense subset
  $A=\big\{\big(\fhi_i,\calH_e(\fhi_i)\big)\,:\, \fhi_i\in
  \dom(\calH_e), i\in\N\big\}$ of
  $\graph\calH_e|_{\dom(\calH_e)}\subset
  H^1_D(\Omega;\R^3)\times\R$. We deduce that there exists a set
  $B\subset (0,T)$ with $|B|=0$ such that \eqref{eq:dpp-S1-pre} holds
  for any $\fhi \in \{\fhi_i\,:\,i\in\N\}$ and any $t\in
  (0,T)\setminus B$. Again by the $2$-growth of $Q$ and the density of
  $A\subset \graph\calH_e|_{\dom(\calH_e)}$ we deduce that \eqref
  {eq:dpp-S1-pre} holds for any $\fhi \in H^1_D(\Omega;\R^3)$ and any
  $t\in (0,T)\setminus B$. This provides \eqref{eq:lem-min-u}.
\end{proof}

\begin{lemma}[Item \ref{it:E2}]
  Let Assumptions \ref{ass:convergence} and \ref{ass:approx-sol} be
  satisfied.  Then \eqref{eq:E2} of Item \ref{it:E2} holds for the
  limit functions $(p,\Sigma)$.
\end{lemma}

\begin{proof}
  We choose a sequence $(\bar u^N)_N$ with $\bar u^N(t)\in \frakE(\bar p^N(t); \bar
  f^N(t))$ as in Lemma \ref{lem:3.5}; in particular, we obtain a weak
  limit $u\in L^2(0,T;H^1_D(\Omega;\R^3))$ such that
  \eqref{eq:lem-conv-uN} holds.

  From \eqref{eq:strong-pN-bar} we deduce that, for almost all $t\in
  (0,T)$, the sequence $(\bar p^N(t))_N$ is strongly convergent in
  $L^2(\Omega)$ to the limit $p(t)$. We can additionally assume that
  the sequence $(\bar p^N(t))_N$ is uniformly bounded in
  $H_0(\Omega,\curl)\cap W^{1,r}(\Omega;\Rnn)$. For such a $t$ we have
  the weak convergence
  \begin{align}
    \bar p^N(t)\,\weakto\, p(t)\quad\text{ in } H_0(\Omega,\curl)\cap
    W^{1,r}(\Omega;\Rnn)\,. \label{eq:weak-pN}
  \end{align}
  
  The approximate solution property \eqref {eq:approx-E2} and property
  \eqref{eq:lem-min-uN} of $\bar u^N(t)$ yield
  \begin{align}
    &\calE\big(\bar p^N(s);\bar f^N(s)\big) \Big|_{s=0}^t+ \int_0^t
    \calR(\del_t \hat p^N(s)) +
    \calR^*(\bar \Sigma^N(s)) \, ds \notag\\
    &\qquad \le -\int_0^t \big\la \partial_t \hat f^N(s), \bar
    u^N(s)\big\ra\, ds + o(1)\,. \label{eq:lem3.6-pf1}
  \end{align}
  The first term on the left-hand side is $\calE_1\big(\bar
  p^N(t);\bar f^N(t)\big) + \calW_p(\bar p^N(t))$. Using the lower
  semi-continuity of Lemma \ref{lem:marginal}, Item
  \ref{it:lem2.5-4a}, the convergence \eqref{eq:weak-pN}, and the
  convexity of $\calW_p$ we deduce the lower semicontinuity of this
  term in the limit,
  \begin{align*}
    \calE\big(p(s);f(s)\big)\Big|_{s=0}^t
    &\leq\, \liminf_{N\to\infty} \calE\big(\bar p^N(s);\bar f^N(s)\big)\Big|_{s=0}^t
  \end{align*}
  for almost all $t\in (0,T)$, where we exploited that the initial
  values are fixed. By the convergences \eqref{eq:SigmaN} and
  \eqref{eq:dtpN-hat}, the growth assumptions \eqref {eq:cond-QR-2}
  and \eqref {eq:cond-QR-3}, and the convexity of $\calR$ and
  $\calR^*$, we also have
  \begin{align*}
    \int_0^t \calR(\del_t p(s)) + \calR^*(\Sigma(s)) \, ds\,\leq\,
    \liminf_{N\to\infty} \int_0^t \calR(\del_t \hat p^N(s)) +
    \calR^*(\bar \Sigma^N(s))\,ds\,.
  \end{align*}
  Finally, \eqref{eq:fN-hat} and \eqref{eq:lem-conv-uN} imply the
  convergence of the right-hand side of \eqref{eq:lem3.6-pf1}. We
  therefore obtain
  \begin{align*}
    &\calE\big(p(s);f(s)\big)\Big|_{s=0}^t 
    + \int_0^t  \calR(\del_t p(s)) + \calR^*(\Sigma(s)) \, ds\\
    &\qquad \leq\, \liminf_{N\to\infty}
    \Big(\calE\big(\bar p^N(s);\bar f^N(s)\big)\Big|_{s=0}^t
    +  \int_0^t  \calR(\del_t \hat p^N(s)) + \calR^*(\bar \Sigma^N(s))\,ds\Big)\\
    &\qquad \leq\, -\int_0^t \big\la \partial_t f(s),  u(s)\big\ra\, ds 
    \leq\, -\int_0^t \inf_{\tilde u\in
      \frakE(p(s),f(s))}\big\la \partial_t f(s), \tilde u\big\ra\, ds\,,
  \end{align*}
  and have verified \eqref{eq:E2} for the limit functions.
\end{proof}

\subsection{Stability for $\delta = 0$}

In this section, we consider the case without gradient-term in the
plastic energy, i.e.\ $\delta=0$ in $\calW_p$ of \eqref
{eq:W-sub-p}. We can treat this case only under further structural
assumptions on the energy: We demand $H_e\equiv 0$ and $H_p\equiv
0$. Our aim is to show a stability result that replaces Proposition
\ref {prop:delta_>0} in the case $\delta=0$.  The existence result of
Theorem \ref {thm:main-eq} will be a consequence of the stability
property.

We proceed along the lines of the case $\delta >0$.  We start with our
assumptions on the approximate solution sequence. The main difference
is that only a weak convergence of the sequences $(\hat p^N)_N$ and
$(\bar p^N)_N$ can be assumed.

\begin{assumption}[Convergence properties,
  $\delta=0$]\label{ass:convergence-d=0}
  We consider a sequence of approximate solutions $(\bar p^N,
  \bar\Sigma^N)_N$ together with a sequence of loads $(\bar f^N)_N$
  that satisfy
  \begin{align}
    \bar p^N \,&\weakto\, p \quad\text{ in }
    L^2(0,T;H_0(\Omega,\curl))\,,
    \label{eq:pN-d=0}\\
    \bar \Sigma^N \,&\weakto\, \Sigma \quad\text{ in }L^2(0,T;L^2(\Omega,\Rnn))\,, 
    \label{eq:SigmaN-d=0}\\
    \bar f^N\,&\to\, f \quad\text{ in }L^2(0,T;L^2(\Omega;\R^3))\,,
    \label{eq:fN-bar-d=0}
  \end{align}
  as $N\to \infty$. For approximations $(\hat p^N)_N$ and $(\hat
  f^N)_N$ we assume, as $N\to \infty$,
  \begin{align}
    \hat p^N \,&\weakto\, p \quad\text{ in }H^1(0,T;L^2(\Omega,\Rnn))\,, 
    \label{eq:dtpN-hat-d=0}\\
    \hat f^N\,&\to\, f \quad\text{ in }H^1(0,T;H^{-1}_D(\Omega;\R^3)) \,.
    \label{eq:fN-hat-d=0}
  \end{align}
  We additionally assume: $(\hat p^N)_N$ and $(\bar p^N)_N$ are
  bounded in $L^\infty(0,T; H_0(\Omega,\curl))$, $(\nabla\cdot
  \bar\Sigma^N)_N$ is bounded in $L^2(0,T; L^2(\Omega))$, $(\bar
  f^N)_N$ is bounded in $L^\infty(0,T; H_D^{-1}(\Omega))$. Finally, we
  impose a weak regularity of $\bar p^N$ in time: We demand
  \begin{equation}
    \label{eq:shifts-time-compactness}
    \sup_N \left\| \bar p^N(. + \rho) - \bar p^N(.)\right\|^2_{L^2(0,T-\rho;L^2(\Omega))} 
    \to 0\text{ as } \rho \to 0\,.
  \end{equation}
  Regarding initial values we assume that $\hat p^N(0) = \bar
  p^N(0)=p(0)=p_0$ and $\hat f^N(0) = \bar f^N(0) = f(0)$ holds for
  all $N\in\N$. We assume that $\bar p^N$, $\bar f^N$, and $\del_t
  \hat f^N$ are simple functions for all $N\in\N$.
\end{assumption}

In the case $\delta=0$ we will impose the same approximate solution
properties as in the case $\delta>0$, i.e.\,those of Assumption \ref
{ass:approx-sol}. We find that the analog of Proposition \ref
{prop:delta_>0} holds.

\begin{proposition}\label{prop:delta_=0}
  Let Assumptions \ref{ass:approx-sol} and \ref{ass:convergence-d=0}
  be satisfied. Then the limiting pair $(p, \Sigma)$ is a weak
  solution of the original problem in the sense of Definition
  \ref{def:solution2}.
\end{proposition}

\begin{proof}
  {\em Step 1: Item \ref{it:R2}.}  As weak limits, the functions $p$
  and $\Sigma$ are in the function spaces of Item \ref{it:R2}:
  $\Sigma\in L^2(0,T;L^2(\Omega; \Rnn))$, $p\in H^{1}(0,T;L^2(\Omega;
  \Rnn)) \cap L^2(0,T;H_0(\Omega,\curl))$.

  \medskip {\em Step 2: Item \ref{it:F2}.}  The proof of Item
  \ref{it:F2} is analogous to the case $\delta>0$, see Lemma \ref
  {lem:item-F2_>0}. The only difference regards the limit procedure of
  the product term, leading to \eqref {eq:F2-eta-pre}: For arbitrary
  $\theta\in C^\infty_c((0,T);\R)$ we claim
  \begin{equation}\label{eq:div-curl-application-1}
    \int_0^T \theta(t) \la \bar \Sigma^N(t), \bar p^N(t)\ra\, dt \to
    \int_0^T \theta(t) \la \Sigma(t), p(t)\ra\, dt
  \end{equation}
  as $N\to \infty$.  This limit is a consequence of the global
  div-curl lemma. The {\em global} div-curl lemma yields not only the
  distributional convergence of a product of weakly convergent
  sequences, but also the convergence of the integral of the
  product. For a global div-curl lemma, one always has to make use of
  boundary conditions. In our case, we know that tangential components
  of $\bar p^N$ vanish on the boundary by $\bar p^N(t) \in
  H_0(\Omega,\curl)$, see \eqref {eq:bdry-p}. For a proof of the
  global div-curl lemma without $t$-dependence see e.g.\,Lemma 6.1 in
  \cite {Schweizer-Friedrichs-2017}.

  For the case with $t$-dependence as in \eqref
  {eq:div-curl-application-1}, we argue as follows. We use a small
  parameter $\rho>0$ and a smooth sequence of symmetric mollifiers
  $\fhi_\rho: \R\to \R$. Functions that are defined on the interval
  $(0,T)$ are always identified with their trivial extension to all of
  $\R$. We claim that the following modification of \eqref
  {eq:div-curl-application-1} is valid:
  \begin{equation}\label{eq:div-curl-application-2}
    \int_0^T \theta(t) 
    \la (\fhi_\rho * \bar \Sigma^N)(t), \bar p^N(t)\ra\, dt \to
    \int_0^T \theta(t) \la (\fhi_\rho * \Sigma)(t), p(t)\ra\, dt\,.
  \end{equation}

  \smallskip {\em Step 2a: Verification of \eqref
    {eq:div-curl-application-2}.}  The sequence $\bar p^N$ is bounded
  in $L^\infty(0,T; L^2(\Omega))$ and, by \eqref
  {eq:shifts-time-compactness}, pre-compact in $L^2(0,T;
  H^{-1}(\Omega))$ \cite[Theorem 1]{Simo87}. We therefore have $\bar
  p^N(t) \to p(t)$ in $H^{-1}(\Omega)$ for almost every $t\in
  (0,T)$. By the boundedness assumptions, this yields $\bar p^N(t)
  \weakto p(t)$ in $L^2(\Omega)$ with $\curl\, \bar p^N(t)$ bounded in
  $L^2(\Omega)$ for almost every $t$.

  The function $\fhi_\rho * \bar \Sigma^N$ is of class $C^1([0,T],
  L^2(\Omega))$ with bounds that are independent of $N$. We therefore
  have $(\fhi_\rho * \bar \Sigma^N)(t) \weakto (\fhi_\rho *
  \Sigma)(t)$ in $L^2(\Omega)$ for every $t\in (0,T)$.  Furthermore,
  $\nabla\cdot (\fhi_\rho * \bar \Sigma^N)(t)$ is bounded in
  $L^2(\Omega)$.  The time-independent global div-curl lemma can be
  applied to the above functions and provides $\theta(t) \la
  (\fhi_\rho * \bar \Sigma^N)(t), \bar p^N(t)\ra\, \to \theta(t) \la
  (\fhi_\rho * \Sigma)(t), p(t)\ra$ for almost every $t$. Dominated
  convergence implies \eqref {eq:div-curl-application-2}.

  \smallskip {\em Step 2b: Verification of \eqref
    {eq:div-curl-application-1}.} Concerning the right hand sides of
  \eqref {eq:div-curl-application-1} and \eqref
  {eq:div-curl-application-2} we observe that $\fhi_\rho * \Sigma \to
  \Sigma$ holds in $L^2(0,T;L^2(\Omega))$.

  Concerning the left hand sides, we observe
  \begin{align*}
    &\int_0^T  \big\la \bar \Sigma^N(t),
    \fhi_\rho * (\theta\bar p^N)(t) - (\theta\bar p^N)(t)\big\ra\, dt\\
    &\quad = \int_\R \int_\R \big\la \bar \Sigma^N(t),
    \fhi_\rho(s)
    \big((\theta \bar p^N)(t-s) - (\theta\bar p^N)(t)\big)\big\ra\, dt\, ds
    \displaybreak[2]\\
    &\quad \le \int_\R \| \bar \Sigma^N \|_{L^2(0,T;L^2(\Omega))}
    \| (\theta \bar p^N)(.-s) - \theta\bar p^N \|_{L^2(0,T;L^2(\Omega))} \fhi_\rho(s) \, ds\,.
  \end{align*}
  For a small parameter $\rho>0$, the right-hand side is small,
  uniformly in $N$:
  \begin{align*}
    &\sup_{|s|<\rho}\| (\theta \bar p^N)(.-s) - \theta\bar p^N \|_{L^2(0,T;L^2(\Omega))}\\
    &\quad \leq\, 2\|\theta\|_{C^1([0,T])}\Big(\sup_{|s|<\rho}\| \bar
    p^N(.-s) - \bar p^N \|_{L^2(0,T;L^2(\Omega))} + \sqrt{\rho}\| \bar
    p^N \|_{L^2(0,T;L^2(\Omega))}\Big)\,.
  \end{align*}
  The smallness of the first contribution is guaranteed by the
  compactness property \eqref{eq:shifts-time-compactness}, the second
  contribution is small by the factor $\sqrt{\rho}$. We obtain that
  the convergence \eqref {eq:div-curl-application-2} provides the
  convergence \eqref {eq:div-curl-application-1}.

  \medskip {\em Step 3: Construction of displacement fields.} In the
  case $\delta>0$, the energy minimizing displacement fields have been
  constructed in Lemma \ref {lem:3.5}. The argument is similar in the
  case $\delta=0$, we sketch here how the four-line calculation in the
  proof of Lemma \ref {lem:3.5} must be altered. 

  We assume that a minimizing sequence $\bar u^N$ is already
  chosen. Let $\fhi\in L^2(0,T; H^1_D(\Omega;\R^3))$ be an arbitrary
  test-function. We claim that we can find a sequence $(\fhi^N)_N$
  with $\fhi^N\in L^2(0,T; H^1_D(\Omega;\R^3))$ such that
  \begin{equation}
    \label{eq:fhi-N}
    \fhi^N \weakto \fhi \text{ in } L^2(0,T; H^1(\Omega))\,
    \, \text{ and }\,
    \nabla \fhi^N - \bar p^N \to \nabla\fhi - p \text{ in } 
    L^2(0,T; L^2(\Omega))\,.
  \end{equation}

  \smallskip
  {\em Step 3a: Conclusion using \eqref {eq:fhi-N}.}  Let us assume
  that we have a sequence $(\fhi^N)_N$ satisfying \eqref
  {eq:fhi-N}. We can calculate, using first the minimality of $\bar
  u^N$, then the definition of $W_e$ and $H_e \equiv 0$, and finally the
  convergence properties of the different sequences:
  \begin{align*}
    &\liminf_N  \int_0^T \theta(t) \left(
    \int_{\Omega} W_e(\nabla \bar u^N(t), \bar p^N(t))
    -  \int_{\Omega} \bar f^N(t)\cdot \bar u^N(t) \right) \,dt \displaybreak[2]\\
    &\qquad \le \liminf_N 
    \int_0^T \theta(t) \left( \int_{\Omega} W_e(\nabla \fhi^N(t), \bar p^N(t)) 
    -  \int_{\Omega} \bar f^N(t)\cdot \fhi^N(t)\right) \,dt \displaybreak[2]\\
    &\qquad = \liminf_N  \int_0^T \theta(t) \left(
      \int_{\Omega} Q(\sym(\nabla \fhi^N(t) - \bar p^N(t))) 
    -  \int_{\Omega} \bar f^N(t)\cdot \fhi^N(t)\right) \,dt\displaybreak[2]\\
    &\qquad =   \int_0^T \theta(t) \left(
    \int_{\Omega} Q(\sym(\nabla \fhi(t) - p(t))) 
    -  \int_{\Omega} f(t)\cdot \fhi(t) \right)\, dt\,.
  \end{align*}
  In the last step we exploited \eqref {eq:fhi-N} and the $2$-growth
  of $Q$ from \eqref{eq:cond-QR-1}.  The rest of the proof is as in
  the case $\delta>0$.

  \smallskip {\em Step 3b: Construction of $\fhi^N$ satisfying \eqref
    {eq:fhi-N}.} We obtain the sequence $(\fhi^N)_N$ from the
  Helmholtz decomposition in space of the function
  \begin{equation}
    q^N := \nabla \fhi + \bar p^N - p\,.
    \label{eq:q-N-def}
  \end{equation}
  The Helmholtz decomposition of $q^N(t)$
  provides a gradient-potential $\fhi^N(t)$ and a curl-potential
  $\Psi^N(t)$ such that
  \begin{equation}
    \label{eq:Helmholz-qN}
    q^N(t) = \nabla \fhi^N(t) + \curl\,\Psi^N(t)\,.
  \end{equation}
  Since $q^N$ is a matrix field $\Omega\to \Rnn$, we apply the usual
  Helmholtz decomposition for vector fields $\Omega\to \R^3$ to each
  row of $q^N$. This yields the desired result, since the $k$-th row
  of $\nabla\fhi$ is $\nabla\fhi_k$, and the $k$-th row of
  $\curl\,\Psi$ is $\curl\Psi_{k,.}$.

  On the potentials, we impose $\nabla\cdot \Psi^N(t) = 0$ (for each
  row $\Psi_{k,.}^N$), and the following boundary conditions: On
  $\Gamma_D$, we demand that the normal component of $\Psi^N(t)$
  vanishes and $\fhi^N(t)|_{\Gamma_D} = 0$, hence $\fhi^N(t) \in
  H^1_D(\Omega)$.  On $\del\Omega \setminus \Gamma_D$, we demand that
  tangential components of $\Psi^N(t)$ vanish.  The existence of the
  two potentials together with $H^1(\Omega)$-estimates is guaranteed
  by the Helmholtz decomposition result, see e.g.\,Theorem 4.2 of
  \cite {Schweizer-Friedrichs-2017} and, for mixed boundary
  conditions, \cite {MR3542004}.

  By the boundedness of the potentials, we may assume the weak
  convergence of $\fhi^N, \nabla \fhi^N, \Psi^N, \nabla \Psi^N$ in the
  space $L^2(0,T; L^2(\Omega))$.  The weak limits of the potentials
  provide a Helmholtz decomposition of the weak limit of $q^N$, which
  is $\nabla\fhi$. Uniqueness of the Helmholtz decomposition implies
  $\fhi^N \weakto \fhi$ and and $\Psi^N\weakto 0$ in $L^2(0,T;
  H^1(\Omega))$.  Relation \eqref {eq:fhi-N} is verified once we show
  $q^N - \nabla\fhi^N \to 0$ strongly in $L^2(0,T; L^2(\Omega))$ or,
  equivalently,
  \begin{equation}
    \xi^N := \curl\,\Psi^N \to 0
    \text{ in } L^2(0,T; L^2(\Omega))\,.\label{eq:2943612}
  \end{equation}

  For the subsequent argument we observe that $q^N(t)$ and $\nabla
  \fhi^N(t)$ have vanishing tangential components on $\Gamma_D$. This
  implies that also $\xi^N(t) = \curl\,\Psi^N(t)$ has vanishing
  tangential components on $\Gamma_D$.
  
  The boundedness of $\Psi^N\in L^2(0,T; H^1(\Omega))$ implies the
  spatial regularity of this sequence. We furthermore know that, for
  small $|\rho|$, $\rho\in \R$, differences $\bar p^N(. + \rho) - \bar
  p^N(.)$ are small in $L^2(0,T; L^2(\Omega))$, independent of $N$, by
  \eqref {eq:shifts-time-compactness}. The same is true for
  $\nabla\fhi$ and $p$. Since the Helmholtz decomposition yields a
  continuous linear map $q^N(t) \mapsto \Psi^N(t)$ from $L^2(\Omega)$
  to $H^1(\Omega)$, this implies that also the sequence $\Psi^N$ has
  small differences $\Psi^N(. + \rho) - \Psi^N(.)$ in $L^2(0,T;
  H^1(\Omega))$. The Fr{\'e}chet-Kolmogorov compactness criterion
  yields the strong convergence $\Psi^N\to 0$ in $L^2(0,T;
  L^2(\Omega))$.
  
  We can now conclude \eqref {eq:2943612}. We first note that the
  sequence $\curl\,\xi^N = \curl\,(q^N - \nabla \fhi^N) =
  \curl\,(\nabla \fhi + \bar p^N - p - \nabla \fhi^N) = \curl\,(\bar
  p^N - p)$ is bounded in $L^2(0,T; L^2(\Omega))$.  We can therefore
  calculate in the limit $N\to \infty$
  \begin{align*}
    \int_0^T\int_\Omega |\xi^N|^2 
    &= \int_0^T\int_\Omega \xi^N\cdot \curl\, \Psi^N
    = \int_0^T\int_\Omega \curl\,\xi^N\cdot \Psi^N \to 0\,.
  \end{align*}
  In this calculation, the boundary conditions for $\xi^N$ and
  $\Psi^N$ allow the integration by parts. We have obtained the strong
  convergence \eqref {eq:2943612}.

  \medskip {\em Step 4: Item \ref{it:E2}.} The proof of Item
  \ref{it:E2} is essentially as in the case $\delta>0$, we only need
  to replace the weak convergence \eqref{eq:weak-pN} by the weak
  convergence in $H_0(\Omega,\curl)$.
\end{proof}

\section{The time-stepping scheme}
\label{sec:tds}

\subsection{Solutions to the discrete problem and estimates}

In this section we construct time-discrete approximations of the
system \eqref{eq:u}--\eqref {eq:p}. With a number $N\in \N$ of time
steps of length $\tau:=\frac{T}{N}$ we discretize the interval $[0,T]$
with
\[ 0 = t_0 < t_1 < \hdots < t_N = T,\quad t_k := k\, \tau,\quad k = 0,
\hdots,N. \] We use a variational scheme to obtain a familiy
$(p_k)_{1\leq k\leq N}$ in the state space
\begin{align*}
  \Xp\,:=\, 
  \begin{cases}
    W^{1,r}_0(\Omega,\Rnn) \quad &\text{ if }\delta>0\,,\\
    H_0(\Omega,\curl) \quad &\text{ if }\delta=0\,.
  \end{cases}
\end{align*}
The functions $p_k\in \Xp$ shall be approximations of the solution
values $p(t_k)$. For $k=0$, we use the initial data $p_0$ as the value
in $t_0 = 0$. The loads are discretized with time averages as
\begin{align}
  \label{eq:f_k}
  f_k := \frac{1}{\tau} \int_{t_{k-1}}^{t_k} f(s)\, ds \quad\text{ for
  } k=2,\dots,N,\quad\text{ and }\ f_0 := f_1 := f(0)\,.
\end{align}
We note that \eqref{ass:f} yields a constant $\Lambda_f>0$,
independent of $N$, such that
\begin{equation}
  \sum_{k=0}^{N} \tau \| f_{k} \|_{L^2(\Omega;\R^3)}^2 
  + \max_{1\leq k\leq N}\|f_k\|^2_{H^{-1}_D(\Omega;\R^3)} 
  + \sum_{k=1}^{N}\tau \Big\|\frac{f_{k}-f_{k-1}}{\tau}\Big
  \|_{H^{-1}_D(\Omega;\R^3)}^2 
  \le \Lambda_f^2\,. \label{eq:est-fN}
\end{equation}

\begin{lemma}[Existence of time-discrete
  approximations]\label{lem:existence-discrete-variational}
  For all $k=1,\ldots,N$ there exists a pair $(p_k,\Sigma_k)\in
  \Xp\times L^2(\Omega;\Rnn)$ such that
  \begin{equation}
    \calE(p_k;f_k)+\calE^*(-\Sigma_k;f_k)
    \,=\, \la -\Sigma_k , p_k\ra\,, \label{eq:td-pk}
  \end{equation}
  and
  \begin{equation}
    \calE(p_k;f_k) +
    \tau\calR\Big(\frac{p_k-p_{k-1}}{\tau}\Big)+\tau\calR^*(\Sigma_k)
    \,\leq\, \calE(p_{k-1};f_{k-1})- \la f_k-f_{k-1},u_{k-1}
    \ra \label{eq:td-energy}
  \end{equation}
  holds for any $u_{k-1}\in \frakE(p_{k-1},f_{k-1})$. Furthermore, we
  have
  \begin{align*}
    \Sigma_k
    \,\in\, \partial\calR\Big(\frac{p_k-p_{k-1}}{\tau}\Big)\,.
  \end{align*}
\end{lemma}

\begin{proof}
  We define a functional $\calR_{\tau}$ by setting
  \begin{align*}
    \calR_{\tau}(q)\,:=\, \int_\Omega \tau
    R\left(\frac{q(x)}{\tau}\right)\,dx\,,
  \end{align*}
  and define $\calG_k:L^2(\Omega,\Rnn)\to \R$ by
  \begin{align}
    \calG_k(p) \,:=\, \calE(p;f_k) + \calR_{\tau}(p-p_{k-1})\quad\text{
      if } p\in \Xp\,, \label{eq:QR}
  \end{align}
  and $\calG_k(p) := +\infty$ if $p\in L^2(\Omega,\Rnn)\setminus \Xp$.
  In order to construct approximations $(p_k)_k$, $1\leq k\leq N$, we
  use the following scheme: In every time step we minimize, given
  $p_{k-1}\in \Xp$, the functional $\calG_k$.

  \smallskip {\em Step 1:} Existence of minimizers. We treat here the
  case $\delta=0$, for $\delta>0$ the proof is easily adapted. By
  Lemma \ref{lem:marginal} Items \ref{it:lem2.5-1} and
  \ref{it:lem2.5-4} (convexity and Lipschitz property of $\calE_1$)
  and Assumption \ref{ass:general} (convexity and growth condition),
  $\calG_k$ is convex and lower semi-continuous. Using the lower bound
  \eqref{eq:E1-bdd-below} for the energy together with
  \eqref{eq:est-fN} and the growth condition \eqref{eq:cond-QR-2} on
  $\calR$ we find, for any $p\in \Xp$,
  \begin{align*}
    \calG_k(p) \,&\geq\, -\|p\|_{L^2(\Omega,\Rnn)}^2 - C\Lambda_f^2 +
    \frac{1}{\tau}c_R\|p-p_{k-1}\|_{L^2(\Omega,\Rnn)}^2\,.
  \end{align*}
  This implies that $\calG_k$ is coercive for any $\tau < c_R$.  The
  direct method of the Calculus of Variations implies
  that 
  a minimizer $p_k\in Xp$ of $\calG_k$ exists.

  \smallskip {\em Step 2:} The minimizing property of $p_k$, the
  continuity of $\calR_\tau$ and subdifferential calculus for convex
  functions 
  imply that
  \begin{align*}
    0\,&\in\, \partial\calG_k(p_k)\,=\, \partial\calE(p_k;f_k)
    +\partial\calR_{\tau}(p_k-p_{k-1})\,.
  \end{align*}
  Therefore, there exists $\Sigma_k\in L^2(\Omega;\Rnn)$ with
  \begin{align}
    -\Sigma_k \,&\in\, \partial\calE(p_k;f_k)\,, \label{eq:td-pk-a}\\
    \Sigma_k \,&\in\, \partial \calR_{\tau}(p_k-p_{k-1})\,=\,
    \partial\calR\Big(\frac{p_k-p_{k-1}}{\tau}\Big)\,. \label{eq:td-pk2-a}
  \end{align}
  Subdifferential calculus provides
  that 
  \eqref{eq:td-pk-a} implies \eqref{eq:td-pk}. We now use
  \eqref{eq:td-pk2-a} and the Fenchel equality, then
  \eqref{eq:td-pk-a} and the defining property of the
  subdifferential. In the last equality, we re-order terms and add a
  zero.
  \begin{align*}
    &\tau \calR\Big(\frac{p_k-p_{k-1}}{\tau}\Big) +\tau\calR^*(\Sigma_k)
    \,=\, \tau \Big\la \Sigma_k,\frac{p_k-p_{k-1}}{\tau}\Big\ra\\
    &\qquad \leq\, \calE(p_{k-1};f_k) -\calE(p_k;f_k) \\
    &\qquad =\, \calE(p_{k-1};f_{k-1}) -\calE(p_k;f_k)+\calE(p_{k-1};f_k)
    -\calE(p_{k-1};f_{k-1})\,.
  \end{align*}
  Item \ref{it:lem2.5-4-1} of Lemma \ref{lem:marginal} yields for any
  $u_{k-1}\in \frakE(p_{k-1},f_{k-1})$
  \begin{align*}
    \tau \calR\Big(\frac{p_k-p_{k-1}}{\tau}\Big) +\tau\calR^*(\Sigma_k)
    \, \leq\, \calE(p_{k-1};f_{k-1})-\calE(p_k;f_k) - \la
    f_k-f_{k-1}, u_{k-1}\ra\,.
  \end{align*}
  This proves \eqref{eq:td-energy}.
\end{proof}

In the following lemma, we treat the case $\delta=0$ and hence $H_e\equiv 0$ and
$H_p\equiv 0$. In this setting, we deduce that the divergence of
$\Sigma_k$ is controlled. Let us provide the idea of the argument: We
have calculated in the introduction for this case $\Sigma = \sigma - 2
\curl\,\curl(p)$. Since the divergence of a curl vanishes, we can
expect $-\nabla\cdot \Sigma = -\nabla\cdot \sigma = f$, which is
controlled. We use $\calQ : L^2(\Omega;\Rnns)\to \R$, defined by
\begin{equation}\label{eq:calQ}
  \calQ(e)\,:=\, \int_\Omega Q(e)\,dx\,.
\end{equation}
Since we assume $H_e\equiv 0$ the elastic energy of an arbitrary
deformation $\fhi\in H^1_D(\Omega,\R^3)$ and of an arbitrary plastic
contribution $q\in L^2(\Omega,\Rnn)$ is given by
\begin{equation*}
  \calW_e(\nabla \varphi,q)\,=\,\calQ(\sym(\nabla \varphi - q))\,=\,
  \int_\Omega Q(\sym(\nabla \varphi - q))\,dx\,.
\end{equation*}

\begin{lemma}\label{lem:H_K=0}
  Assume $\delta = 0$ with $H_e\equiv 0$ and $H_p\equiv 0$. Then, for
  any $1\leq k\leq N$, there holds 
  \begin{align}
    -\nabla \cdot \Sigma_k \,&=\, f_k
    \,. \label{eq:div-Sigma-k}
  \end{align}
\end{lemma}

\begin{proof}
  We analyze the distribution $\eta:= -\Sigma_k - 2 \curl\curl (p_k)$.
  From \eqref{eq:td-pk} we know $-\Sigma_k \in \del \calE(p_k; f_k)$.
  Subdifferential calculus (e.g.\,\cite[Theorem 9.5.4]{AtBM14}) yields
  $\eta\in \partial \calE_1(p_k;f_k)$. The definition of the
  subdifferential yields
  \begin{equation}\label{eq:82362394}
    \calE_1(p_k+\nabla\psi;f_k) \,\geq\, \calE_1(p_k;f_k) 
    + \la \eta,\nabla\psi\ra
  \end{equation}
  for all $\psi\in H^1_D(\Omega,\R^3)$.  We can evaluate the left hand
  side, arguing with $\tilde\fhi = \fhi - \psi$,
  \begin{align*}
    \calE_1(p_k+\nabla\psi;f_k)\,&=\, \inf_{\varphi\in
      H^1_D(\Omega,\R^3)}
    \Big( \calQ(\sym(\nabla\varphi - p_k - \nabla\psi)) 
    -\la \varphi,f_k\ra \Big)\\
    &= \inf_{\tilde\varphi\in H^1_D(\Omega,\R^3)} \Big(
    \calQ(\sym(\nabla\tilde\varphi-p_k)) - \la \tilde\varphi,f_k\ra \Big)
    -\la \psi,f_k\ra\\
    &= \calE_1(p_k;f_k)-\la \psi,f_k\ra\,.
  \end{align*}
  Inserting into \eqref {eq:82362394} yields
  \begin{align*}
    0 \,\geq\, \la \eta,\nabla\psi\ra + \la \psi,f_k\ra \quad\text{ for
      all }\psi\in H^1_D(\Omega,\R^3)\,,
  \end{align*}
  and hence $-\nabla\cdot \eta + f_k =0$. By definition of $\eta$, and
  since the divergence of a curl vanishes, this yields the claim of
  \eqref{eq:div-Sigma-k}.
\end{proof}

\begin{lemma}[A priori bounds for the time-discrete
  solutions]\label{lem:timestep-apriori}
  Let the load $f$ satisfy \eqref{ass:f} and let $(f_k)_k$ be defined
  by \eqref {eq:f_k} such that \eqref {eq:est-fN} is satisfied. Then
  there exists a constant $C = C(\Lambda_f,
  \|p_0\|_{L^2(\Omega,\Rnn)})$, independent of $N$, such that the
  sequence of time-discrete solutions satisfies the a priori estimate
  \begin{equation}\label{eq:timestep-apriori}
    \max_{k} \calW(\nabla u_k,p_k)
    + \sum_k\tau \left\{ \calR\Big(\frac{p_k - p_{k-1}}{\tau}\Big) 
    + \calR^*(\Sigma_k)\right\}
    \le C\,,
  \end{equation}
  where $u_k\in\frakE(p_k,f_k)$, $1\leq k\leq N$, are chosen
  arbitrarily.  In particular, we have
  \begin{align}
     \max_{1\leq k\leq N} \|p_k\|_{H_0(\Omega,\curl)}^2 
     + \max_{1\leq k\leq N} \int_\Omega \delta |\nabla p_k|^r 
     + \sum_{1\leq k\leq N}\tau\int_\Omega \left\{
     \Big|\frac{p_k-p_{k-1}}{\tau}\Big|^2 +|\Sigma_k|^2\right\} \,\leq\,
    C\,. \label{eq:td-estimates}
  \end{align}
\end{lemma}

\begin{proof} We choose $1\leq k_0\leq N$ arbitrary and take the sum
  of \eqref{eq:td-energy} over $k=1, \dots ,{k_0}$. We obtain
  \begin{equation}\label{eq:4.2-1}
    \begin{split}
      &\calE(p_{k_0};f_{k_0})
      + \sum_{k=1}^{k_0}\tau \left\{ \calR\Big(\frac{p_k-p_{k-1}}{\tau}\Big)
        +\calR^*(\Sigma_k) \right\}\\
      &\qquad \leq\, \calE(p_{0};f_{0}) -\sum_{k=1}^{k_0}\int_\Omega
      u_{k-1} \cdot (f_k-f_{k-1})\,,
    \end{split}
  \end{equation}
  where the functions $u_{k-1}\in \frakE(p_{k-1},f_{k-1})$,
  $k=1,\dots,k_0$, are arbitrary.

  We estimate the second term on the right-hand side of
  \eqref{eq:4.2-1}. By our choice $f_0=f_1$, using the bound
  \eqref{eq:est-fN} on $(f_k)_k$, we obtain for arbitrary $\lambda>0$
  \begin{align}
    &\Big|\sum_{k=1}^{{k_0}}\int_\Omega u_{k-1}\cdot (f_{k}-f_{k-1})\,dx \Big|\notag\\
    &\quad \leq\,  
    \Big(\sum_{k=2}^{k_0}\tau 
    \Big\|\frac{f_{k}-f_{k-1}}{\tau}\Big\|_{H^{-1}_D(\Omega;\R^3)}^2\Big)^{\frac{1}{2}}
    \Big(\sum_{k=2}^{{k_0}}\tau \|u_{k-1}\|_{H^1(\Omega;\R^3)}^2\Big)^{\frac{1}{2}}\notag\\
    &\quad \leq\, \frac{\Lambda_f^2}{4\lambda} 
    +\lambda\, k_0\, \tau 
    \max_{1\leq k\leq {k_0}-1} \|u_k\|_{H^1(\Omega;\R^3)}^2 \notag\\
    &\quad \leq\, \lambda T \max_{1\leq k\leq k_0-1}
    \|p_k\|_{L^2(\Omega;\Rnn)}^2 +
    C(T,\lambda,\Lambda_f)\,, \label{eq:4.2-1a}
  \end{align}
  where we have used the estimate for $u_k$ from \eqref{eq:est-minE1}
  in the last step. We can now obtain an $N$-independent bound for
  energies. We use \eqref{eq:E1-bdd-below} in the first inequality,
  then \eqref{eq:4.2-1} together with \eqref{eq:4.2-1a}, with the
  parameter $\lambda>0$ unchanged,
  \begin{align*}
    &\frac{1}{2}\calW_e(\nabla u_{k_0},p_{k_0}) + \calW_p(p_{k_0}) 
    + \sum_{k=1}^{k_0}\tau\left\{\calR\Big(\frac{p_k-p_{k-1}}{\tau}\Big)
    +\calR^*(\Sigma_k) \right\} \notag\\
    &\quad \leq\, \calE(p_{k_0};f_{k_0}) 
    + \sum_{k=1}^{k_0}\tau \left\{ \calR\Big(\frac{p_k-p_{k-1}}{\tau}\Big)
    +\calR^*(\Sigma_k) \right\} + \lambda \|p_{k_0}\|^2 
    + C  \notag\\
    &\quad \leq\,  \calE(p_{0};f_0) +  \lambda (T+1) \max_{1\leq k\leq k_0}
    \|p_k\|_{L^2(\Omega;\Rnn)}^2 + C\,,
  \end{align*}
  where $C = C(T,c_Q,\lambda,\Omega,\Gamma_D,\Lambda_f)$.  This
  implies
  \begin{align}
    &\max_{1\leq k\leq N}\calW(\nabla u_k,p_k) +
    \sum_{k=1}^{N}\tau\left\{ \calR\Big(\frac{p_k-p_{k-1}}{\tau}\Big)
    +\calR^*(\Sigma_k) \right\} \notag\\
    &\quad \leq\, 3\calE(p_{0};f_0) + 3\lambda (T+1) \max_{1\leq k\leq
      N} \|p_k\|_{L^2(\Omega;\Rnn)}^2 +  C\,. \label{eq:4.2-2}
  \end{align}
	
  Our next aim is to find bounds for $(p_k)_k$. We consider an
  arbitrary $k\leq N$ and start with the elementary triangle
  inequality
  \begin{equation*}
    \|p_{k}\|_{L^2(\Omega;\Rnn)} \,\leq\,
    \sum_{j=1}^{k}\|p_{j}-p_{j-1}\|_{L^2(\Omega;\Rnn)} +
    \|p_0\|_{L^2(\Omega;\Rnn)}\,.
  \end{equation*}
  Taking the square and using the coercivity \eqref{eq:cond-QR-2} of
  $\calR$, we obtain
  \begin{align}
    \|p_{k}\|_{L^2(\Omega;\Rnn)}^2 \,&\leq\, 
    2k\sum_{j=1}^{k}\|p_{j}-p_{j-1}\|_{L^2(\Omega;\Rnn)}^2 
    + 2\|p_0\|_{L^2(\Omega;\Rnn)}^2 \notag\\
    &\leq\, \frac{2k\tau^2}{c_R} \sum_{j=1}^{k}
    \calR\Big(\frac{p_j-p_{j-1}}{\tau}\Big) + 2\|p_0\|_{L^2(\Omega;\Rnn)}^2 
    \notag\displaybreak[2]\\
    &\leq\,
    C(T,c_R)\sum_{j=1}^{N}\tau\calR\Big(\frac{p_j-p_{j-1}}{\tau}\Big)
    + 2\|p_0\|_{L^2(\Omega;\Rnn)}^2\,. \label{eq:4.2-3}
  \end{align}
  Hence, by choosing $\lambda>0$ sufficiently small in
  \eqref{eq:4.2-2}, we obtain
  \begin{equation}
    \max_{1\leq k\leq N}\calW(\nabla u_k,p_k) 
    + \frac{1}{2}\sum_{k=1}^{N}\tau\left\{ \calR\Big(\frac{p_k-p_{k-1}}{\tau}\Big)
    +\calR^*(\Sigma_k) \right\} \,  \leq\, C\,, \label{eq:4.2-7}
  \end{equation}
  with $C = C(T,c_R,c_Q,\Omega,\Gamma_D,\Lambda_f,p_0)$, and hence
  \eqref {eq:timestep-apriori}.  Using \eqref{eq:cond-QR-2} and
  \eqref{eq:4.2-3} we deduce \eqref{eq:td-estimates}.
\end{proof}

\subsection{The time-continuous limit}

The next step in our construction is to introduce interpolations of
the time-discrete approximate solutions.  For $N\in\N$, we have
constructed time-discrete values $p_k = p_k^N \in X$ and $\Sigma_k =
\Sigma_k^N \in L^2(\Omega,\Rnn)$, see Lemma \ref
{lem:existence-discrete-variational}. We define the piecewise constant
and left-continuous interpolation $\bar p^N : [0,T] \to X$ and the
piecewise affine and continuous interpolation $\hat p^N : [0,T] \to X$
by
\begin{align*}
  \bar p^N(t) &:= p_k\quad \quad\text{ for } 
  t\in (t_{k-1},t_{k}],1\leq k\leq N,\qquad \quad \bar p^N(0):=p_0\\
  \hat p^N(t) &:= (1-\mu) p_{k-1} + \mu p_{k}\quad\text{ for } t=
  (1-\mu) t_{k-1} + \mu t_{k},\ \mu \in [0,1],1\leq k\leq N\,.
\end{align*}
Similarly, we define $\bar \Sigma^N$ and $\bar f^N$ as the piecewise
constant left-continuous interpolations of $(\Sigma_k^N)_k$ and
$(f_k^N)_k$. The function $\hat f^N$ is defined slightly differently,
namely with a time-shift (we set $f_{N+1}:=f_N$):
\begin{align*}
  \hat f^N(t) &:= (1-\mu) f_{k} + \mu f_{k+1}\quad\text{ for } t= 
  (1-\mu) t_{k-1} + \mu  t_{k},\ \mu \in[0,1],1\leq k\leq N\,.
\end{align*}

By the previous results we easily obtain a priori bounds for $\bar
p^N$, $\bar \Sigma^N$, $\partial_t \hat p^N$.
\begin{lemma}[Estimates for the
  interpolations]\label{lem:estimates-discrete}
  Under Assumption \ref{ass:general} there exists $C>0$ independent of
  $N$ such that the piecewise constant functions satisfy
  \begin{equation}
    \| \bar p^N \|_{L^\infty(0,T; H_0(\Omega,\curl))}+ \| \bar
    \Sigma^N \|_{L^2(0,T; L^2(\Omega;\Rnn))} +
    \sqrt[r]{\delta}\|\nabla\bar
    p^N\|_{L^\infty(0,T;L^{r}(\Omega;\Rnn))} \le
    C\,.  \label{eq:aprioritimediscrete}
  \end{equation}
  The piecewise affine functions satisfy
  \begin{equation}
    \| \hat p^N \|_{H^1(0,T; L^2(\Omega;\Rnn))} + \| \hat p^N
    \|_{L^\infty(0,T; H_0(\Omega,\curl))}+
    \sqrt[r]{\delta}\|\nabla\hat
    p^N\|_{L^\infty(0,T;L^{r}(\Omega;\Rnn))}\,\leq\,
    C\,.  \label{eq:aprioritimediscrete2}
  \end{equation}
  We have the time-shift property
  \begin{equation}
    \sup_N \left\| \bar p^N(. + \rho) - \bar p^N(.) 
    \right\|^2_{L^2(0,T-\rho;L^2(\Omega;\Rnn))} \to 0\quad \text{
      as }\quad \rho \to 0\,. \label{eq:shifts-time-compactness-sec3}
  \end{equation}  
  Finally, in the case $\delta=0$, for $H_e\equiv 0$ and $H_p\equiv
  0$, with $\Lambda_f$ from \eqref {eq:est-fN}:
  \begin{equation}\label{eq:div-Sigma-est}
    \|\nabla \cdot
    \bar\Sigma^N\|_{L^2(0,T; L^2(\Omega,\R^3))}\,\leq\, \Lambda_f\,.
  \end{equation}
\end{lemma}

\begin{proof}
  Lemma \ref{lem:timestep-apriori} provides the estimates \eqref
  {eq:aprioritimediscrete} and \eqref {eq:aprioritimediscrete2} for
  $\bar p^N$, $\bar \Sigma^N$, and $\hat p^N$.  Lemma \ref{lem:H_K=0}
  provides \eqref {eq:div-Sigma-est} for $\nabla \cdot \bar\Sigma^N$.
  
  To prove \eqref{eq:shifts-time-compactness-sec3} we first calculate
  for $\rho=\tau_j:=j\tau$, $j\in \N$,
  \begin{align*}
    \int_0^{T-\rho}\|\bar p^N(t+\rho)-\bar
    p^N(t)\|_{L^2(\Omega)}^2\,dt\,&=\,
    \sum_{k=1}^{N-j}\tau\int_\Omega | p_{k+j} - p_k |^2\\
    = \sum_{k=1}^{N-j}\tau\int_\Omega \Big|\sum_{i=k+1}^{k+j}(p_i-p_{i-1})\Big|^2
    &\leq\, \sum_{k=1}^{N-j}\tau j \sum_{i=k+1}^{k+j} \tau^2 \int_\Omega
    \Big|\frac{p_i-p_{i-1}}{\tau}\Big|^2\, \leq\, C \rho^2\,,
  \end{align*}
  where we used the a priori estimate \eqref{eq:td-estimates} in the
  last step.

  We now consider an arbitrary shift $\rho \in (\tau_j, \tau_{j+1}]$,
  $0< j\in \N$. For $t$ lying in an interval $(t_k,t_{k+1}]$, the
  number $t +\rho$ lies either in the interval $(t_{k+j},t_{k+j+1}]$
  or in the interval $(t_{k+j+1},t_{k+j+2}]$.  We hence obtain with a
  triangle inequality
  \begin{equation*}
    | \bar p^N(t + \rho) - \bar p^N(t) | \le
    | \bar p^N(t + \tau_{j+1}) - \bar p^N(t) | +
    | \bar p^N(t + \tau_{j+2}) - \bar p^N(t) | \,.
  \end{equation*}
  The above inequalities yield \eqref
  {eq:shifts-time-compactness-sec3} for $\bar p^N$.
\end{proof}

Lemma \ref{lem:estimates-discrete} implies some compactness
properties.

\begin{lemma}[Convergence of time-discrete approximations]\label{lem:tds-cpct}
  There exists a subsequence $N\to\infty$ (not relabeled) and
  functions $p\in L^\infty(0,T;H_0(\Omega,\curl))\cap
  H^1(0,T;L^2(\Omega;\Rnn))$ and $\Sigma\in L^2(0,T; L^2(\Omega;\Rnn))$ such
  that
  \begin{align}
    \bar p^N\,&\weakstar\, p\quad\text{ in }L^\infty(0,T;H_0(\Omega,\curl))\,, 
    \label{eq:conv-bar_pN}\\
    \hat p^N\,&\weakstar\, p\quad\text{ in }L^\infty(0,T;H_0(\Omega,\curl))\,, 
    \label{eq:conv-hat_pN}\\
    \partial_t\hat p^N\,&\weakto\, \partial_tp\quad\text{ in }L^2(0,T;L^2(\Omega,\Rnn))\,, 
    \label{eq:conv-dt_pN}\\
    \bar \Sigma^N\,&\weakto\, \Sigma\quad\text{ in
    }L^2(0,T;L^2(\Omega,\Rnn))\,. \label{eq:conv-SigmaN}
  \end{align}
  For $\delta>0$ we have additionally $p\in L^\infty(0,T; W^{1,r}
  (\Omega,\Rnn))$ and, for $s=\frac{3r}{3-r}>2$ and arbitrary $1\leq
  q<\infty$
  \begin{align}
    \bar p^N\,&\weakstar\, p\quad\text{ in }L^\infty(0,T;W^{1,r}(\Omega,\Rnn))\,, 
    \label{eq:conv-bar_pN_H1}\\
    \hat p^N\,&\weakstar\, p\quad\text{ in
    }L^\infty(0,T;W^{1,r}(\Omega,\Rnn))\,,
    \label{eq:conv-hat_pN_H1}\\
    \bar p^N\,&\to\, p\quad\text{ in }L^q(0,T;L^s(\Omega,\Rnn))\,, 
    \label{eq:conv-bar_pN_strong}\\
    \hat p^N\,&\to\, p\quad\text{ in
    }L^q(0,T;L^s(\Omega,\Rnn))\,. \label{eq:conv-hat_pN_strong}
  \end{align}
  Moreover, we have
  \begin{align}
    \bar f^N\,&\to\, f \quad\text{ in }L^2(0,T;L^2(\Omega;\R^3))\,,
    \label{eq:fN-bar-1}\\
   \hat f^N\,&\to\, f \quad\text{ in }H^1(0,T;H^{-1}_D(\Omega;\R^3)) \,.
    \label{eq:fN-hat-1}
  \end{align}
\end{lemma}

\begin{proof}
  The a priori estimates \eqref{eq:aprioritimediscrete} and
  \eqref{eq:aprioritimediscrete2} allow to select a subsequence
  $N\to\infty$ and to find limits $p$ and $\Sigma$ as in the claim of
  the Lemma, such that \eqref{eq:conv-hat_pN}, \eqref{eq:conv-dt_pN}
  and \eqref{eq:conv-SigmaN}
  hold. 
  Furthermore, for a limit function $\bar p\in
  L^\infty(0,T;H_0(\Omega,\curl))$, there holds $\bar p^N\weakstar
  \bar p$ in $L^\infty(0,T;H_0(\Omega,\curl))$.

  The sequence $\hat p^N$ can be regarded as a sequence in the space
  $L^2(0,T; H_D^{-1}(\Omega))$. The Lions--Aubin Lemma \cite[Lemma
  7.7]{Roubicek-book} implies that $(\hat p^N)_N$ is pre-compact in
  this space. We therefore have the strong convergence $\hat p^N \to
  p$ in $L^2(0,T; H_D^{-1}(\Omega))$.  The strong convergence allows
  to compare the two interpolations, see \cite[Lemma 3.2]{MR2653755}
  or \cite[Lemma 11.3]{Schw13}; we conclude that $\bar p^N$ has the
  same limit and converges also strongly, $\bar p^N \to \bar p = p$ in
  $L^2(0,T; H_D^{-1}(\Omega))$.  In particular, we find
  \eqref{eq:conv-bar_pN}.

  In case that $\delta>0$ we can apply \eqref{eq:aprioritimediscrete}
  and conclude that $\nabla \hat p^N$ is uniformly bounded in
  $L^\infty(0,T;L^r(\Omega;\R^3))$. Since the only constant function
  in $W^{1,r}(\Omega;\R^3)\cap H_0(\Omega,\curl)$ is zero, by
  \cite[Sec.~6.16]{Alt12} a Poincar{\'e} inequality holds in this
  space and we deduce that $\hat p^N$ is uniformly bounded in
  $L^\infty(0,T;W^{1,r}(\Omega;\R^3))$. We conclude by a similar
  identification argument as above that \eqref{eq:conv-bar_pN_H1} and
  \eqref{eq:conv-hat_pN_H1} hold. 

  The estimate \eqref{eq:aprioritimediscrete2} together with the
  Lions--Aubin Lemma implies the compactness of the sequence $(\hat
  p^N)_N$ in $L^q(0,T;L^s(\Omega,\Rnn))$ and hence \eqref
  {eq:conv-hat_pN_strong}. Comparison of the two interpolations (as
  above, using \cite{MR2653755} or \cite{Schw13}) yields
  \eqref{eq:conv-bar_pN_strong}.

  The statements on $\bar f^N$ and $\hat f^N$ are elementary
  approximation properties for discretizations of a given function
  $f$.
\end{proof}

\begin{proof}[Proof of Theorem \ref{thm:main-gr} and Theorem 
  \ref{thm:main-eq}]
  The discrete solution property \eqref{eq:td-pk} yields that almost
  everywhere in $(0,T)$
  \begin{equation}
    \calE(\bar p^N;\bar f^N) + \calE^*(-\bar \Sigma^N;\bar f^N) \,=\,
    \la -\bar\Sigma^N, \bar p^N \ra\,. \label{eq:2-N}
  \end{equation}

  The energy inequality for the time-discrete approximate solutions
  was derived in \eqref{eq:4.2-1}. We claim that that inequality
  implies, for almost every $t\in (0,T)$,
  \begin{gather}
    \begin{aligned}
      &\calE(\bar p^N;\bar f^N)(t) + \int_0^t \left\{
        \calR\big(\partial_t \hat p^N\big)
        +\calR^*(\bar \Sigma^N)\right\}\,ds \\
      &\qquad\leq\, \calE(p_{0};f_0) - \int_0^t\inf_{\tilde u^N \in
        \frakE(\bar p^N(s);\bar f^N(s))} \la \partial_t\hat f^N(s),
      \tilde u^N \ra\,ds + o(1) \label{eq:3-N}
    \end{aligned}
  \end{gather}
  as $N\to \infty$.

  Indeed, let $u_k\in \frakE(p_k,f_k)$ be chosen arbitrarily. We
  consider first a time instance $t = t_{k_0} \in \tau \N$. For such a
  time instance $t$, the relation \eqref {eq:3-N} coincides with
  \eqref{eq:4.2-1} except for the very last time interval; the
  difference of the right hand sides is $\int_\Omega u_{k_0} \cdot
  (f_{k_0 + 1} - f_{k_0})$.

  Let us now consider an arbitrary time instance. For $t\in
  (t_{k_0-1},t_{k_0}]$ the left hand side of \eqref {eq:3-N} is not
  larger than the left hand side of \eqref {eq:4.2-1}, since $R$ and
  $R^*$ are non-negative. The difference of the right hand sides is
  \begin{equation*}
    \frac{t-t_{k_0-1}}{\tau} \int_\Omega u_{k_0} \cdot (f_{k_0 + 1} - f_{k_0}) 
    = o(1)\,,
  \end{equation*}
  for almost every $t$.  The smallness is a consequence of the
  following facts: (i) uniform bound $t-t_{k_0 - 1} \leq \tau$, (ii)
  uniform (in $k_0$, $N$) boundedness of $u_{k_0}$ in $H^1_D(\Omega)$,
  and (iii) the continuous embedding of $H^1(0,T;H^{-1}(\Omega))$ in
  $C^{0,\alpha}([0,T];H^{-1}(\Omega))$ for all $0<\alpha<\frac{1}{2}$
  and the uniform (in $k_0$, $N$) estimate
  \begin{equation*}
    \|f_{k+1}-f_k\|_{H^{-1}(\Omega)}\,
    =\, \Big\|\frac{1}{\tau}
    \int_{t_k}^{t_{k+1}}\big(f(s)-f(s-\tau)\big)\,ds\Big\|_{H^{-1}(\Omega)}\,
    \leq\, \|f\|_{C^{0,\alpha}([0,T];H^{-1}(\Omega))}\tau^\alpha.
  \end{equation*}

  The two relations \eqref {eq:2-N} and \eqref {eq:3-N} imply that the
  approximate solutions satisfy the solution properties of Assumption
  \ref{ass:approx-sol}.

  Lemma \ref{lem:tds-cpct} provides the convergence properties of
  Assumptions \ref{ass:convergence} and \ref{ass:convergence-d=0} in
  the two cases $\delta > 0$ and $\delta = 0$. We emphasize that, in
  the case $\delta=0$, with \eqref {eq:shifts-time-compactness-sec3}
  we have verified \eqref {eq:shifts-time-compactness} and with \eqref
  {eq:div-Sigma-est} we have verified the boundedness of $\nabla\cdot
  \bar\Sigma^N$.  We can apply the stability results of Propositions
  \ref{prop:delta_>0} and \ref {prop:delta_=0}. They yield that
  $(p,\Sigma)$ is a generalized solution of the visco-plasticity
  system.
\end{proof}

\appendix

\section{The marginal functional}
\label{sec.marginal-solutions}

We discuss some properties of the marginal functional $\calE_1$.
\begin{lemma}\label{lem:marginal}
  The functional $\calE_1(\cdot;f) :L^2(\Omega;\Rnn)\to\R$ from
  \eqref{eq:reduced-energy-1} has the following properties.
  \begin{enumerate}
  \item\label{it:lem2.5-1} For any $f\in H^{-1}_D(\Omega;\R^3)$ the
    functional $\calE_1(\cdot;f)$ is convex.
  \item\label{it:lem2.5-2} For any $p\in L^2(\Omega;\Rnn)$ and any
    $f\in H^{-1}_D(\Omega;\R^3)$ we have
    \begin{equation}
      \calE_1(p;f) \leq
      C(1+\|p\|^2_{L^2(\Omega;\Rnn)})\,. \label{eq:bound-E1}
    \end{equation}
  \item\label{it:lem2.5-2-1} For any $p\in L^2(\Omega;\Rnn)$, $f\in
    H^{-1}_D(\Omega;\R^3)$, $u\in \frakE(p,f)$, and $\lambda>0$ there
    exist $c_\lambda=c_\lambda(C_Q,\Omega,\Gamma_D) > 0$ and
    $C_\lambda=C_\lambda(C_Q,\Omega,\Gamma_D)$ such that
    \begin{equation}
      \calE_1(p;f) \geq \frac{1}{2}\calW_e(\nabla u,p) +
      c_\lambda\|u\|_{H^1_D(\Omega;\R^3)}^2- \lambda \|p\|^2_{L^2(\Omega)} -
      C_\lambda
      \|f\|_{H^{-1}_D(\Omega;\R^3))}^2\,. \label{eq:E1-bdd-below}
    \end{equation}
  \item\label{it:lem2.5-3} For any $p\in L^2(\Omega;\Rnn)$ and any
    $f\in H^{-1}_D(\Omega;\R^3)$ there exists an $u\in \frakE(p,f)$.
    Any $u\in \frakE(p,f)$ satisfies
    \begin{equation}
      \|u\|_{H^1_D(\Omega;\R^3)} \,\leq\, C \big(1+
      \|p\|_{L^2(\Omega;\Rnn)} +
      \|f\|_{H^{-1}_D(\Omega;\R^3)}\big)\,. \label{eq:est-minE1}
    \end{equation}
  \item\label{it:lem2.5-4-1} For any $p\in L^2(\Omega;\Rnn)$, any
    $f,g\in H^{-1}_D(\Omega;\R^3)$ and any $u\in \frakE(p,g)$ we have
    \begin{align}
      \calE_1(p;f)-\calE_1(p;g) \,\leq\, -\la f-g, u \ra
      \,. \label{eq:lem2.5-4-1}
    \end{align}
  \item\label{it:lem2.5-4} The map $\calE_1: L^2(\Omega;\Rnn)\times
    H^{-1}_D(\Omega;\R^3)\to\R$ is locally Lipschitz continuous. More
    precisely, for any $\Lambda_0$ there exists $C(\Lambda_0,Q)$ such
    that
    \begin{equation*}
      \big|\calE_1(p;f)-\calE_1(q;g)|\,\leq\, C(\Lambda_0,Q)\Big(\|p-q\|_{L^2(\Omega)}
      +\|f-g\|_{H_D^{-1}(\Omega)}\Big)
    \end{equation*}
    for any $f,g,p,q$ with 
    \begin{equation}\label{eq:locally-Lipschitz-bound}
      \|f\|_{H^{-1}_D(\Omega;\R^3)} + \|g\|_{H^{-1}_D(\Omega;\R^3)} +
      \|p\|_{L^2(\Omega)} +\|q\|_{L^2(\Omega)}\,\leq\,\Lambda_0\,.
    \end{equation}
  \item\label{it:lem2.5-4a} For any sequence $(p_j,f_j)_j$ and any
    $(p,f)$ in $L^2(\Omega;\Rnn)\times H^{-1}_D(\Omega;\R^3)$ with
    $p_j\weakto p$ in $L^2(\Omega;\Rnn)$ and $f_j\to f$ in
    $H^{-1}_D(\Omega;\R^3)$ we have
    \begin{equation}
      \calE_1(p;f) \,\leq\, \liminf_{j\to\infty}
      \calE_1(p_j;f_j)\,. \label{eq:marginal-lsc}
    \end{equation}
  \end{enumerate}
\end{lemma}

\begin{proof}
  In the following, for any $p\in L^2(\Omega;\Rnn)$, we use the
  shortcut $p^s:=\sym(p)$.

  \noindent {\em Item \ref{it:lem2.5-1}:} Let $p,q\in
  L^2(\Omega;\Rnn)$ and $0<\lambda<1$ be arbitrary. From the convexity
  of $Q$ and $H_e$ we deduce
  \begin{align*}
    &\hspace*{-3mm}(1-\lambda)\calE_1(p;f) + \lambda\calE_1(q;f)\\
    = & \inf_{\fhi,\psi\in H^1_D(\Omega;\R^3)} \Big[\int_{\Omega}\Big(
    (1-\lambda)W_e(\nabla \fhi, p)+\lambda W_e(\nabla
    \psi, q)\Big) - \Big\langle f, ((1-\lambda)\fhi+\lambda\psi\big)\Big\rangle \Big]\\
    \geq & \inf_{\fhi,\psi\in H^1_D(\Omega;\R^3)}
    \Big[\int_{\Omega}\Big( Q\big((1-\lambda)\,(\nabla^s\fhi-p^s)
    +\lambda\,(\nabla^s\psi-q^s)\big)\\
    &\qquad\qquad + H_e\big((1-\lambda)\nabla^s\fhi
    +\lambda\nabla^s\psi\big)\Big)
    - \Big\langle f, \big((1-\lambda)\fhi+\lambda\psi\big)\Big\rangle \Big]\\
    = & \inf_{\tilde\fhi\in H^1_D(\Omega;\R^3)} \Big[\int_{\Omega}
    \Big( Q\big(\nabla^s\tilde\fhi-((1-\lambda)p^s+\lambda q^s))\big) 
    + H_e\big(\nabla^s\tilde\fhi\big)\Big) - \big\langle f, \tilde\fhi\big\rangle \Big]\\
    =&\, \calE_1\big((1-\lambda)p+\lambda q;f)\,.
  \end{align*}
  This proves the convexity of $\calE_1(\cdot;f)$. 

  \noindent {\em Item \ref{it:lem2.5-2}:} We use $u \equiv 0$ as a
  competitor to estimate the energy from above.  The growth assumption
  on $Q$ implies
  \begin{equation*}
    \calE_1(p;f) \,\leq\, \int_{\Omega} \big(Q(\sym\ p)+H_e(0)\big)
    \,\leq\, C(Q,H_e)\, (1+\|p\|^2_{L^2(\Omega)}).
  \end{equation*}

  \noindent {\em Item \ref{it:lem2.5-2-1}:} In any Hilbert space and
  for any $\mu > 0$ holds $\frac{1}{1+\mu}\|a\|^2 -2\langle a,b\rangle
  + (1+\mu)\|b\|^2\geq 0$; this inequality can be rearranged as
  $\|a-b\|^2 \geq \frac{\mu}{1+\mu} \|a\|^2 -\mu\|b\|^2$. Applying the
  latter inequality with $\mu=\frac{2\lambda}{c_Q}$, using the growth
  assumptions on $Q$ and Korn's inequality \cite[Korollar
  25.6]{Schw13}, we deduce that for any $\varphi\in
  H^1_D(\Omega;\R^3)$
  \begin{align}
    &\int_{\Omega} Q(\sym(\nabla \varphi-p))
    -2 \big|\big\langle f, \varphi\big\rangle \big| \notag\\
    &\quad \geq\,
    c_Q\|\sym (\nabla \varphi -p) \|^2_{L^2(\Omega)} 
    - 2\|f\|_{H^{-1}_D(\Omega;\R^3)}\| \varphi \|_{H^{1}_D(\Omega;\R^3)} \notag\\
    &\quad \geq \frac{2c_Q\lambda}{c_Q+2\lambda}\|\nabla^s \varphi \|^2_{L^2(\Omega)} -
    2\lambda \|p\|^2_{L^2(\Omega)} - \frac{c_Q\lambda}{c_Q+2\lambda}\|\nabla^s \varphi
    \|^2_{L^2(\Omega)}  \notag\\
    &\qquad -
    C(C_Q,\lambda,\Omega,\Gamma_D)\|f\|_{H^{-1}_D(\Omega;\R^3)}^2 \notag\\
    &\quad \geq \frac{c_Q\lambda}{c_Q+2\lambda}\|\nabla^s \varphi \|^2_{L^2(\Omega)} -
    2\lambda \|p\|^2_{L^2(\Omega)} -C(C_Q,\lambda,\Omega,\Gamma_D)\|f\|_{H^{-1}_D(\Omega;\R^3)}^2\,.
    \label{eq:E1-bdd-below1}
  \end{align}
  Now consider any $u\in \frakE(p,f)$. The definition of $\calE_1$
  yields $\calE_1(p;f) \ge \calW_e(\nabla u, p) - \la f, u \ra$ and we
  obtain from $H_e\ge 0$ and \eqref{eq:E1-bdd-below1}
  \begin{align*}
    &\calE_1(p;f)-\frac{1}{2}\calW_e(\nabla u,p) \\
    &\quad \geq
    \frac{1}{2}\int_{\Omega} Q(\sym(\nabla u-p))
    - \langle f, u\rangle \notag\\
    &\quad \geq \frac{c_Q\lambda}{2(c_Q+2\lambda)}\|\nabla^s u \|^2_{L^2(\Omega)} -
    \lambda \|p\|^2_{L^2(\Omega)} 
    -C(C_Q,\lambda,\Omega,\Gamma_D)\|f\|_{H^{-1}_D(\Omega;\R^3)}^2\,.
  \end{align*}
  By Korn's inequality, this yields \eqref{eq:E1-bdd-below}.

  \noindent {\em Item \ref{it:lem2.5-3}:} We consider the functional
  $\tilde\calE:H^1_D(\Omega;\R^3)\to\R\cup\{+\infty\}$, given by
  \begin{equation*}
    \tilde\calE(\fhi)\,:=\, \int_{\Omega} W_e(\nabla \fhi, p) -
    \big\langle f, \fhi\big\rangle\,.
  \end{equation*}
  The properties of $Q$ and $H_e$, in particular, their convexity,
  imply that $\tilde\calE$ is lower semicontinuous with respect to
  weak convergence in $H^1_D(\Omega;\R^3)$. By
  \eqref{eq:E1-bdd-below1}, the functional $\tilde\calE$ is also
  coercive.  The direct method of the Calculus of Variations ensures
  the existence of a minimizer $u\in H^1_D(\Omega;\R^3)$ of
  $\tilde\calE$, hence the existence of $u\in \frakE(p,f)$.
  
  The estimates \eqref{eq:bound-E1} and \eqref{eq:E1-bdd-below} yield
  \eqref{eq:est-minE1}.

  \noindent {\em Item \ref{it:lem2.5-4-1}:} By definition of $\calE_1$
  and by the minimization property $u\in \frakE(p,g)$ we have
  \begin{align*}
    \calE_1(p;f)-\calE_1(p,g) \,&\leq\, \calW_e(\nabla u,p) - \la
    f,u\ra  -\Big(  \calW_e(\nabla u,p) - \la g,u\ra\Big)\\
    &=\, -\la f-g, u\ra\,.
  \end{align*}
  
  \noindent {\em Item \ref{it:lem2.5-4}:} We consider a bound
  $\Lambda_0 > 0$ and study functions $f, g\in H^{-1}_D(\Omega;\R^3)$
  and $p, q\in L^2(\Omega;\Rnn)$ satisfying \eqref
  {eq:locally-Lipschitz-bound}.  We deduce from Item \ref{it:lem2.5-3}
  that for $\Lambda = C(1+ \Lambda_0)$
  \begin{equation*}
    \calE_1(p;f) \,=\, \inf_{\fhi\in
      H^1_D(\Omega;\R^3),\|\fhi\|\leq\Lambda} \int_{\Omega}
    W_e(\nabla \fhi, p) - \big\langle f, \fhi \big\rangle \,,
  \end{equation*}
  and accordingly for $\calE_1(q;g)$.

  We use once more the functional $\calQ$ of \eqref {eq:calQ}.  By the
  quadratic growth assumption \eqref{eq:cond-QR-1} on $Q$, this
  functional is Lipschitz continuous on $B(0,\Lambda_0+\Lambda)
  \subset L^2(\Omega;\Rnns)$ with some Lipschitz constant $L>0$
  depending only on $\Lambda_0+\Lambda$, see e.g.\,\cite[Theorem
  4.47]{FoLe07}. Therefore, for any $\fhi\in H^1_D(\Omega;\R^3)$ with
  $\|\fhi\|_{H^1_D(\Omega;\R^3)}\leq\Lambda$ we have
  \begin{align*}
    &\Big|\int_{\Omega} Q(\sym(\nabla\fhi- p)) - \big\langle f,
    \fhi\big\rangle - \Big(\int_{\Omega} Q(\sym(\nabla\fhi- q))
    - \big\langle g, \fhi\big\rangle \Big)\Big|\\
    &\qquad \leq\, L\|p-q\|_{L^2(\Omega)} +
    \Lambda\|f-g\|_{H^{-1}_D(\Omega;\R^3)}\,.
  \end{align*}
  This implies
  \begin{align*}
    &\calE_1(p;f)-\calE_1(q;g) \\
    &\qquad \leq\, \sup_{\fhi\in
      H^1_D(\Omega;\R^3),\|\fhi\|\leq\Lambda} \Big(\int_{\Omega}
    W_e(\nabla\fhi, p) - \big\langle f, \fhi\big\rangle
    -\Big[\int_{\Omega} W_e(\nabla\fhi, q) 
    - \big\langle g, \fhi \big\rangle\Big]\Big)\\
    &\qquad =\, \sup_{\fhi\in H^1_D(\Omega;\R^3),\|\fhi\|\leq\Lambda}
    \Big(\int_{\Omega} Q(\sym(\nabla\fhi- p)) - Q(\sym(\nabla\fhi- q))
    - \big\langle f-g, \fhi \big\rangle\Big)\\
    &\qquad \leq\, L\|p-q\|_{L^2(\Omega)} +
    \Lambda\|f-g\|_{H^{-1}_D(\Omega;\R^3)}\,,
  \end{align*}
  and similarly
  \begin{equation*}
    \calE_1(p;f)-\calE_1(q;g) \geq\, -L\|p-q\|_{L^2(\Omega)} -
    \Lambda\|f-g\|_{H^{-1}_D(\Omega;\R^3)}\,.
  \end{equation*}
  These inequalities prove the Lipschitz-continuity of $\calE_1$.

  \noindent {\em Item \ref{it:lem2.5-4a}:} The functional
  $\calE_1(\cdot;f)$ is lower semicontinuous under weak convergence in
  $L^2(\Omega;\R^3)$ because of Lipschitz continuity and the convexity
  of Item \ref{it:lem2.5-1}. Since the sequence $(p_j, f_j)_j$ is
  uniformly bounded in $L^2(\Omega;\R^3)\times H^{-1}_D(\Omega;\R^3)$
  we deduce from Item \ref{it:lem2.5-4-1}
  \begin{align*}
    \liminf_{j\to\infty} \calE_1(p_j;f_j) \,
    &\geq\, \liminf_{N\to\infty} \calE_1(p_j;f) 
    + \liminf_{j\to\infty}\Big(\calE_1(p_j;f_j)-\calE_1(p_j;f)\Big)\\
    &\geq\, \calE_1(p;f) -
    \limsup_{j\to\infty}C\|f_j-f\|_{H^{-1}_D(\Omega;\R^3)} =\,
    \calE_1(p;f)\,.
  \end{align*}
  This concludes the proof of the lemma.
\end{proof}

We can now show the equivalence of the solution concepts.

\begin{proof}[Proof of Proposition \ref{prop:solution_concepts}]

  \noindent {\em Item 1:} Let $(p,\Sigma)$ be a solution according to
  Definition \ref{def:solution2}.  Our aim is to show that there
  exists a solution $(u,p,\Sigma)$ according to Definition
  \ref{def:solution1}.

  The existence statement before \eqref{eq:est-minE1} yields that, for
  almost every $s\in (0,T)$, there exists $u(s)\in \frakE(p(s),f(s))$.

  We claim that the map $s\mapsto u(s)$ is measurable.  We choose
  sequences $(p^N)_N$, $(f^N)_N$ of simple functions $p^N:(0,T)\to
  L^2(\Omega,\Rnn)$ and $f^N:(0,T)\to H^{-1}_D(\Omega,\R^3)$ such that
  $p^N(t)\to p(t)$ in $L^2(\Omega,\Rnn)$ and $f^N(t)\to f(t)$ in
  $H^{-1}_D(\Omega,\R^3)$ for almost all $t\in (0,T)$.  For $t\in
  (0,T)$ we choose an element $u^N(t)\in\frakE(p^N(t),f^N(t))$, such
  that $u^N:(0,T)\to H^1_D(\Omega,\R^3)$ is a simple function. By the
  uniform boundedness and lower-semicontinuity properties
  \eqref{eq:est-minE1} and \eqref{eq:marginal-lsc} we deduce that
  $u^N(t)\weakto u(t)$ in $H^{1}_D(\Omega,\R^3)$ for almost all $t\in
  (0,T)$. Hence, with respect to the weak topology in
  $H^1_D(\Omega,\R^3)$, the function $s\mapsto u(s)$ can pointwise
  almost everywhere be approximated by simple functions. The Pettis
  measurability theorem \cite[Theorem 2.104]{FoLe07} yields that
  $u:(0,T)\to H^{1}_D(\Omega,\R^3)$ is measurable.
	
  The estimate \eqref{eq:est-minE1} implies $u\in
  L^2(0,T;H^1_D(\Omega;\R^3))$, hence $(u,p,\Sigma)$ meets the
  regularity requirements \ref{it:R1} of Definition
  \ref{def:solution1}. The stability property \ref{it:S1} is a
  consequence of $u(t)\in \frakE(p(t),f(t))$ for almost all $t\in
  (0,T)$.
  
  We next prove that \ref{it:F2} implies \ref{it:F1}.  We consider
  arbitrary functions $f\in H^{-1}_D(\Omega)$, $p,q\in
  L^2(\Omega;\Rnn)$, and minimizers $u\in\frakE(p,f)$,
  $v\in\frakE(q,f)$. By definition of the energies and by minimality,
  there holds
  \begin{align*}
    &\calE(p;f) -\calE(q;f)  \\
    &\qquad =\, \calE_1(p;f) -\calE_1(q;f) +\calW_p(p)-\calW_p(q)\\
    &\qquad =\, \calW_e(\nabla u,p)-\la f,u\ra 
    -\calW_e(\nabla v,q)+\la f,v\ra +\calW_p(p)-\calW_p(q)\\
    &\qquad \geq\, \calW_e(\nabla u,p) -\calW_e(\nabla u,q)
    +\calW_p(p)-\calW_p(q)\\
    &\qquad =\, \calW(\nabla u,p) -\calW(\nabla u,q)\,.
 \end{align*}
 Hence, for arbitrary $\Sigma\in L^2(\Omega;\Rnn)$,
 \begin{align*}
   \calE(p;f) +\calE^*(-\Sigma;f) \,&=\, \sup_{q\in L^2(\Omega;\Rnn)}
   \Big(\calE(p;f) -\calE(q;f) +\la -\Sigma,q\ra\Big)  \\
   &\geq\, \sup_{q\in L^2(\Omega;\Rnn)}\Big(\calW(\nabla u,p) 
   -\calW(\nabla u,q) +\la -\Sigma,q\ra\Big) \\
   &=\, \calW(\nabla u,p) +\calW^*(\nabla u,-\Sigma)\,.
 \end{align*}
 Inserting into property \ref{it:F2}, we find that for almost all
 $t\in (0,T)$
 \begin{equation*}
   \calW(\nabla u,p) +\calW^*(\nabla u,-\Sigma)\,\leq\, \la-\Sigma,p\ra\,.
 \end{equation*}
 Since the opposite inequality always holds, we obtain equality and
 hence \ref{it:F1}.

 In order to conclude \ref{it:E1}, we first note that $u(t)\in
 \frakE(p(t),f(t))$ implies
  \begin{equation*}
    \calE (p(t);f(t))
    \,=\,  \calW(\nabla u(t),p(t))  - \la f(t),u(t)\ra\,.
  \end{equation*}
  With this equality, \ref{it:E2} implies \ref{it:E1}; the infimum in
  \eqref{eq:E2} is attained by the uniqueness assumption.

  \smallskip
  \noindent {\em Item 2:}
  To prove that regular weak solutions $(u,p,\Sigma)$ are strong
  solutions, we first observe that the stability property \ref{it:S1}
  implies (under the additional regularity assumptions) the
  Euler--Lagrange equation
  \begin{align*}
    0\,=\, -\nabla\cdot \left(\sym \nabla_{F} W_e(\nabla
      u(t),p(t))\right) - f(t)
  \end{align*}
  for almost all $t\in (0,T)$.  Choosing $\sigma$ as in
  \eqref{eq:sig}, this is the balance of forces \eqref{eq:u}.  The
  back-stress relation \ref{it:F1} implies directly \eqref{eq:Sigma}.

  It remains to derive the flow rule \eqref{eq:p}.  We write the first
  term on the left hand side of \ref{it:E1} as an integral over its
  time derivative
  \begin{align}
    &\frac{d}{ds} \left(\calW(\nabla u(s), p(s)) -\int_\Omega
      f(s)\cdot u(s)\right)
    \notag\displaybreak[2]\\
    &\qquad =\, \big\la \partial_t p(s),\nabla_p \calW(\nabla u(s),
    p(s))\big\ra
    -\langle \partial_t f(s), u(s)\rangle\notag\displaybreak[2]\\
    &\qquad =\, -\big\la \partial_t p(s),\Sigma(s)\big\ra
    -\langle \partial_t f(s), u(s)\rangle\,,
    \label{eq:98723452}
  \end{align}
  where terms containing $\del_t u(s)$ cancel by balance of forces.
  Inserting into \ref{it:E1} yields
  \begin{align*}
    \int_0^T \big\la \partial_t p(s), -\Sigma(s)\big\ra +
    \calR(\partial_t p(s)) +\calR^*(\Sigma(s))\,ds \,\leq\, 0\,.
  \end{align*}
  By definition of $\calR^*$, the integrand is nonnegative, hence
  \begin{align*}
    \big\la \partial_t p(s),-\Sigma(s)\big\ra + \calR(\partial_t p(s))
    +\calR^*(\Sigma(s)) \,=\, 0
  \end{align*}
  for almost all $s\in (0,T)$.  This yields \eqref{eq:p}.

  \smallskip
  \noindent {\em Item 3:}
  We now consider a strong solution $(u,\sigma,p,\Sigma)$ to
  \eqref{eq:u}--\eqref{eq:p}.  By the balance of forces \eqref{eq:u}
  and \eqref{eq:sig}, $u(t)$ is a critical point of the convex map
  $\varphi\mapsto \calW( \nabla\varphi, p(t)) - \la f(t),\varphi\ra$
  for almost all $t\in (0,T)$. It hence satisfies the minimality
  \ref{it:S1}. Property \eqref{eq:Sigma} of the back-stress $\Sigma$
  yields \ref{it:F1}. Balance of forces \eqref{eq:u} and
  \eqref{eq:sig} allows to calculate as in \eqref {eq:98723452}.  We
  obtain, using \eqref{eq:p},
  \begin{align*}
    &\frac{d}{ds} \left(\calW(\nabla u(s), p(s)) -\int_\Omega
      f(s)\cdot u(s)\right)
    + \calR(\partial_t p(s)) +\calR^*(\Sigma(s)) \displaybreak[2]\\
    &\qquad =\,-\langle \partial_t f(s), u(s)\rangle\,.
  \end{align*}
  Integrating this equality implies \ref{it:E1}.
\end{proof}


\end{document}